\documentclass[11pt]{amsart}
\usepackage{amsthm,amsmath,amsfonts,amssymb}
\usepackage[numbers]{natbib}
\usepackage[colorlinks,citecolor=blue,urlcolor=blue]{hyperref}
\usepackage{subcaption}
\usepackage{graphicx}

\theoremstyle{plain}
\newtheorem{theorem}{Theorem}
\newtheorem{lemma}[theorem]{Lemma}
\newtheorem{proposition}[theorem]{Proposition}
\newtheorem{remark}[theorem]{Remark}
\newtheorem{corollary}[theorem]{Corollary}

\theoremstyle{remark}

\newtheorem{assumption}{Assumption}

\newcommand{\E}{{\mathbb E}}
\newcommand{\R}{{\mathbb R}}
\newcommand{\N}{{\mathbb N}}
\newcommand{\dd}{{\rm d}}
\newcommand{\ii}{\ensuremath{\mathrm{i}}}
\newcommand{\e}{\mathrm{e}}
\providecommand{\norm}[1]{\left\lVert#1\right\rVert}

\begin{document}

\title[A splitting scheme for stochastic Schr\"odinger equations]{Analysis of a splitting scheme for a class of nonlinear stochastic Schr\"odinger equations}

\author{Charles-Edouard Br\'ehier}

\author{David Cohen}

\address[Charles-Edouard Br\'ehier]{{Universit\'e de Pau et des Pays de l'Adour, E2S UPPA, CNRS, LMAP},
            {Pau},
            {64013},
            {France}}
\email{charles-edouard.brehier@univ-pau.fr}

\address[David Cohen]{{Mathematical Sciences, Chalmers University of Technology and University of Gothenburg},
            {Gothenburg},
            {41296},
            {Sweden}}
\email{david.cohen@chalmers.se}

\begin{abstract}
We analyze the qualitative properties and the order of convergence of a splitting scheme for a class of
nonlinear stochastic  Schr\"odinger equations driven by additive noise.
The class of nonlinearities of interest includes nonlocal interaction cubic nonlinearities.
We show that the numerical solution is symplectic and preserves the expected mass for all times (trace formula).
On top of that, for the convergence analysis, some exponential moment bounds
for the exact and numerical solutions are proved. This enables us to provide strong orders of convergence
as well as orders of convergence in probability and almost surely.
Finally, extensive numerical experiments
illustrate the performance of the proposed numerical scheme.
\end{abstract}

\maketitle
{\small\noindent
{\bf AMS Classification (2020).} 60-08. 60H15. 60M15. 65C20. 65C30. 65C50. 65J08}

Stochastic Schr\"odinger equations.
Splitting integrators.
Strong convergence.
Geometric numerical integration. Trace formulas.

\section{Introduction}\label{intro}

Deterministic Schr\"{o}dinger  equations are widely used within physics,
plasma physics or nonlinear optics, see for instance
\cite{MR1696311, agrawal2007nonlinear, MR1087449, MR1273317}.
In certain physical situations it may be appropriate to incorporate
some randomness into the model. One possibility is to add a driving random force
and obtain a stochastic partial differential equation (SPDE) of the form
$$
\ii\frac{\partial u}{\partial t}(x,t)=\Delta u(x,t)+F(x,u)+\xi(x,t),
$$
considered for $x\in\mathbb{T}^d$, the $d$-dimensional torus, with periodic boundary conditions.
The nonlinearity $F$ and the white noise $\xi$ are described in details below.
See Equation~\eqref{sse} for the formulation of this problem as a stochastic evolution equation.
The stochastic nonlinear Schr\"odinger equations are used as nonlocal
models of wave propagation in several physical applications, see for example~\cite{
AStochasticNonlinearSchrodingerEquation,effectonnoise,MR1954077, MR2135313,MR2268663,Theoreticalnumericalaspectsstochasticnonlinear,MR1425880}
and references therein for further details and applications.
The nonlinearities we shall consider encompass
for instance the cases of an external potential $F(x,u)=V(x)u$ and
of a nonlocal interaction cubic nonlinearity $F(x,u)=(V\star |u|^2)u$. The latter nonlocal interaction function is defined
using the convolution of a sufficiently regular interaction kernel $V$ with the density function $|u|^2$.
The stochastic models that we consider have a structure similar to those of the deterministic Schr\"odinger--Poisson equations or Hartree equations, which
are simplified models for fundamental equations in quantum transport where nonlocal interaction terms appear,
see for instance \cite{MR1087449,MR1273317} and references therein. Such nonlinearities are also used in modeling deterministic problems arising in quantum physics, chemistry, materials sciences, and biology, see for instance~\cite{MR3348493,bddl}.
Nonlocal interaction terms for deterministic models
can be obtained by mean-field limits starting from
many body systems, see for instance \cite{MR3475664,MR1890644}.
We are not aware of similar derivation for stochastic models. Regarding the numerical
treatment of nonlocal interaction terms for Schr\"odinger equations with white noise
dispersion, we refer to the recent article \cite{MR4400428}.

It is worth mentioning that the error analysis below requires stringent regularity assumptions on the interaction kernel $V$. These conditions are not satisfied, for instance, for the Schr\"odinger--Poisson case (where $F(x,u)=\pm(\frac{1}{4\pi|x|}\star |u|^2)u$), for the cubic case (where $F(x,u)=\pm|u|^2u$) or the more general power-law case (where $F(x,u)=\pm|u|^{2\sigma}u$ with $\sigma>0$).
Note that some of the techniques presented below may be used to study numerical schemes for such nonlinearities. This is left for possible future works.

Let us now review the literature on temporal discretizations of stochastic Schr\"{o}dinger
equations, first for equations driven by an It\^o noise. In \cite{MR2268663},
a Crank--Nicolson scheme is studied for the stochastic Schr\"odinger equation with
regular coefficients. First order of strong convergence, resp. rate one half is obtained
in the case of additive noise, resp. multiplicative It\^o noise.
In addition, convergence in probability as well as almost-surely are studied for the case
of a power-law nonlinearity. Finally, in the case of smooth and bounded nonlinearity, weak order one is proved in \cite{MR2268663}.
Observe that the numerical scheme from \cite{MR2268663} is implicit.
The references \cite{Numericalsimulationfocusingstochastic, Barton-Smith_et_al-2005-Numerical_Methods_for_Partial_Differential_Equations}
present thorough numerical simulations and numerically study the effect of noise
in the stochastic Schr\"odinger equation with a power-law nonlinearity.
The work \cite{MR3021492} provides a strong convergence analysis of a splitting strategy to
the variational solution of a stochastic Schr\"odinger equation with regular coefficients.
The recent article \cite{MR3771721} proves strong convergence of an exponential integrator for stochastic Schr\"odinger equations
with regular coefficients. In addition, longtime behaviors of the numerical solutions of a linear model is investigated.
The paper \cite{MR3884775} provides a convergence rate of the weak error under noise discretizations of
some Schr\"odinger equations. The work \cite{MR3986278} shows convergence
in probability of a stochastic (implicit) symplectic scheme
for stochastic nonlinear Schr\"odinger equations with
quadratic potential and an additive noise.
The article \cite{MR3605170} proves weak error estimates for a spatial as well as temporal numerical approximation
of the stochastic cubic Schr\"odinger equation with damping and trace-class noise.
The recent work \cite{MR3826675} proves strong rate $1/2$ as well as weak rate $1$ of a splitting scheme when applied to
a damped stochastic cubic Schr\"odinger equation with linear multiplicative trace-class noise and large enough damping term.
The very recent preprint \cite{cui21} provides several convergence results for a structure-preserving splitting strategy when
applied to nonlinear stochastic Schr\"odinger equations with damping term and multiplicative noise.
There is a vast literature on the numerical analysis of stochastic nonlinear Schr\"odinger equations
with Stratonovich noise and power-law nonlinearities. Without being exhaustive,
we mention the work \cite{MR2036364} which proves convergence in probability of the Crank--Nicolson scheme applied
to such equations with a spatially correlated noise. The work \cite{MR3072234} provides
first order of convergence in probability and in the almost-sure
sense for a splitting scheme applied to the cubic case. The article \cite{MR3498984} shows
strong order one of convergence in the local sense of the $\theta$-scheme when applied to the stochastic cubic
Schr\"odinger equation with a trace-class noise. Strong rate of convergence one for a finite difference approximation
of the stochastic cubic Schr\"odinger equation with a colored noise is shown in \cite{MR3670034}.
The recent article \cite{MR3912762} shows optimal strong order of convergence of a splitting Crank--Nicolson scheme when applied
to the spectral Galerkin spatial discretization of the stochastic cubic Schr\"odinger equation with trace-class noise.

In the present work, we shall analyze a splitting strategy for an efficient time
integration of a class of nonlinear stochastic Schr\"odinger equations, see Equation~\eqref{sse}.
In a nutshell, the main idea of splitting integrators is to decompose the vector field of the
original evolution equation in several parts, such that the arising subsystems are exactly (or easily)
integrated. We refer interested readers to \cite{MR2840298,MR3642447,MR2009376} for details
on splitting schemes for ordinary and partial differential equations.
Splitting schemes are also very popular and efficient numerical integrators for
stochastic differential equations: we refer the interested reader to the articles
\cite{MR1860968,MR2608370,MR3362507,MR3511359,MR3570281,bct21,cv21,MR4422320}, the list of references is not exhaustive.
For stochastic partial differential equations, splitting schemes
have been studied for instance in the following works  \cite{m06,MR2646103,MR3021492,MR3119724,MR3617573,MR3607207,MR3736651,MR3912762,
MR3839068,MR4019051,MR4132896,MR4263224,MR4278943,MR4400428,bcg22}. The splitting scheme considered
in this publication is given by equation \eqref{S1}.

Despite the fact that splitting schemes are widely used for an efficient time
integration of deterministic Schr\"odinger-type equations,
see for instance \cite{MR1739100,MR1921908,MR2429878,MR2472388,MR2628827,MR3072234,MR3702412},
we are not aware of a numerical analysis of such integrators approximating mild solutions
of nonlinear stochastic Schr\"odinger equations driven by an additive noise.
In the present work, we intend to fill this gap for a class of nonlinear SPDEs
and the main results of this paper are the following:
\begin{itemize}
\item bounds for the exponential moments of the mass of the exact and numerical solutions (Theorem~\ref{thm-expm});
\item a kind of longtime stability, namely a so-called trace formula for the mass, of the exact and numerical solutions (Proposition~\ref{prop:mass});
\item preservation of symplecticity for the exact and numerical solutions (Proposition~\ref{prop:symp});
\item strong convergence estimates (with order) of the splitting scheme (Theorem~\ref{th:conv}), as well as orders of convergence in probability and almost surely  (Corollary~\ref{cor:pas}).
\end{itemize}
Observe that, since the nonlinearity in the class of stochastic Schr\"odinger equation considered here may not be
globally Lipschitz, we employ the exponential moments estimates mentioned above to obtain strong rates of convergence,
see Propositions~\ref{prop:conv1/2}~and~\ref{prop:conv1}. In these propositions,
we consider moments of the error multiplied by an exponential discounting factor,
and obtain the expected rate of convergence for this quantity.
This technique is similar to the approach in~\cite{MR4079431}.
Combining those estimates with the above exponential moment bounds (Theorem~\ref{thm-expm}) to
remove the exponential discounting factor, we can then obtain Theorem~\ref{th:conv}. Note finally,
that the choice of a splitting strategy is crucial in obtaining exponential moment bounds for the numerical solution.

Let us mention that in Corollary~\ref{cor:pas} one obtains orders of convergence $\frac12$ and $1$ in the sense of convergence in probability
and almost sure convergence (depending on a regularity parameter) which are expected to be optimal,
whereas in Theorem~\ref{th:conv} the orders of convergence in the strong sense may not be optimal.
This is due to the exponential moment bounds used in the proofs.

We begin the exposition by introducing some notations, present our main assumptions and provide several moment bound estimates
for the exact solution to the considered SPDE.
We then present the splitting scheme and study some geometric properties of the exact and numerical solutions
in Section~\ref{sec-split}.
The main results of this publication are presented in Section~\ref{sec-conv}.
In particular, exponential moments in the $L^2$ norm of the exact and numerical solutions are given,
as well as several convergence results. More involved and technical proofs of results needed for
convergence estimates are provided in Section~\ref{sec-proofs}.
Various numerical experiments illustrating
the main properties of the splitting scheme
when applied to stochastic Schr\"odinger equations
driven by It\^o noise are given in Section~\ref{sec-num}.
The paper ends with an appendix containing proofs of auxiliary results.

We use $C$ to denote a generic constant, independent of the time-step size of the numerical scheme,
which may differ from one place to another.

\section{Setting}\label{notation}

In this work, we consider the following class of stochastic nonlinear Schr\"odinger equations
\begin{equation}
\begin{aligned}
&\ii\, \mathrm{d}u(t)  = \Delta u(t)\, \mathrm{d}t + F(u(t))\,\mathrm{d}t+\,\alpha\mathrm{d}W^Q(t), \\
&u(0) = u_0,
\end{aligned}
\label{sse}
\end{equation}
where the unknown $\bigl(u(t)\bigr)_{t\ge 0}$ is a stochastic process with values in the Hilbert space $L^2=L^2(\mathbb{T}^d)$ of square integrable complex-valued functions defined on the $d$-dimensional torus $\mathbb{T}^d$. Details concerning the regularity and growth properties of the nonlinearity $F$ and the covariance operator $Q$ are provided below. In addition, $\alpha>0$ is a real parameter measuring the size of the noise $W^Q$. The initial condition $u_0\in L^2$ is deterministic, however the
results below can be adapted to random initial conditions, satisfying appropriate integrability conditions, using a standard conditioning argument. The space $L^2$ is equipped with the norm $\norm{\cdot}_{L^2}$, where for all $u,v\in L^2$,
\[
\norm{u}_{L^2}^2=\langle u,u\rangle,\quad \langle u,v\rangle=\int_{\mathbb{T}^d}\bar{u}(x)v(x)\,\dd x.
\]
The Sobolev spaces $H^1=H^1(\mathbb{T}^d)$ and $H^2=H^2(\mathbb{T}^d)$ are Hilbert spaces, and the associated norms are denoted by $\norm{\cdot}_{H^1}$ and $\norm{\cdot}_{H^2}$. The notation $H^0=L^2$ will also be used below. For $\sigma\in\{0,1,2\}$,
let also $\norm{\cdot}_{\mathcal{C}^ \sigma}$ denote the norm in the Banach space
$\mathcal{C}^\sigma=\mathcal{C}^\sigma(\mathbb{T}^d)$ of functions
of class $\mathcal{C}^\sigma$ defined in $\mathbb{T}^d$.

Solutions of~\eqref{sse} are understood in the mild sense:
\begin{equation}\label{mild}
u(t)=S(t)u_0-\ii\int_0^t S(t-s)F(u(s))\,\mathrm{d}s-\ii\alpha\int_0^tS(t-s)\,\mathrm{d}W^Q(s),
\end{equation}
where $S(t)=\e^{-\ii t\Delta}$. Let us state the following result
(see e.g.~\cite[Lemma $3.1$]{Amass-preservingsplittingscheme} and \cite[Appendix  A.$1$]{MR2268663}).
\begin{lemma}\label{lemma1}
The linear operator $-\ii\Delta$ generates a group $\bigl(S(t)\bigr)_{t\in\R}$ of isometries of $L^2$, such that for all $\sigma\in\{0,1,2\}$, all $u\in H^\sigma$, and all $t\ge 0$, one has
\[
\norm{S(t)u}_{H^\sigma}=\norm{u}_{H^\sigma}.
\]
In addition, for $\sigma\in\{1,2\}$, there exists $C_\sigma\in(0,\infty)$ such that for all $u\in H^\sigma$ and all $t\ge 0$,
\[
\norm{\bigl(S(t)-I\bigr)u}_{L^2}\le C_\sigma t^{\frac{\sigma}{2}}\norm{u}_{H^\sigma}.
\]
\end{lemma}

The Wiener process $W^Q$, with covariance operator $Q$, in the SPDE \eqref{sse} is defined by
\begin{equation*}
W^Q(t)=\sum_{k\in\N}\gamma_k\beta_k(t)e_k,
\end{equation*}
where $\bigl(e_k\bigr)_{k\in\N}$ is a complete orthonormal system of $L^2$, $\bigl(\beta_k\bigr)_{k\in\N}$ is a sequence of independent real-valued standard Wiener processes on a stochastic basis $(\Omega,\mathcal F, \mathbb P, (\mathcal F(t))_{t\ge0})$,
and $\bigl(\gamma_k\bigr)_{k\in\N}$ is a sequence of complex numbers such that $\displaystyle \sum_{k\in\N}|\gamma_k|^2<\infty$. The linear operators $Q$ and $Q^{\frac12}$ are defined by $Qe_k=\gamma_k^2e_k$ and $Q^{\frac12}e_k=\gamma_ke_k$, for all $k\in\N$.

For a linear operator $\Psi$ from $H^\sigma$ to $H^\sigma$, and any complete orthonormal system
$\left(\varepsilon_k\right)_{k\in\N}$ of $H^\sigma$, we define
\[
\norm{\Psi}_{\mathcal{L}_2^\sigma}^2=\sum_{k\in\N}\norm{\Psi \varepsilon_k}_{H^\sigma}^2.
\]
This definition is independent of the choice of the orthonormal system.

With this notation, $\norm{Q^{\frac12}}_{\mathcal{L}_2^\sigma}^2=\displaystyle\sum_{k\in\N}|\gamma_k|^2\norm{e_k}_{H^\sigma}^2$
(whenever the sum is finite).

We now set the assumptions on the spatial Sobolev regularity of the noise as well as on the nonlinearity in the
stochastic Schr\"odinger equation \eqref{sse}
required to prove well-posedness for the SPDE~\eqref{sse}, to prove $H^1$-regularity of the solution,
and to show strong convergence of order $1/2$ of the proposed splitting integrator in Section~\ref{sec-conv}.
\begin{assumption}\label{as1}
One has
$$
\norm{Q^{\frac12}}_{\mathcal{L}_2^1}^2=\sum_{k\in\N}|\gamma_k|^2\norm{e_k}_{H^1}^2<\infty.
$$
The nonlinearity $F$ satisfies $F(u)=V[u]u$ for all $u\in L^2$, where $V\colon u\in L^2\mapsto V[u]\in\R$ is a real-valued mapping.
Furthermore, it is assumed that $V[u_1]=V[u_2]$ if $|u_1|=|u_2|$ (i.\,e. the potential $V$ is a function of the modulus).

In addition to the above, assume that the mapping $F$ is locally Lipschitz continuous with at most cubic growth: there exists $C_F\in(0,\infty)$ and $K_F\in(0,\infty)$ such that for all $u_1,u_2\in L^2$, one has
\begin{equation}\label{eq:KF}
\norm{F(u_2)-F(u_1)}_{L^2}\le \left(C_F+K_F(\norm{u_1}_{L^2}^2+\norm{u_2}_{L^2}^2)\right)\norm{u_2-u_1}_{L^2}.
\end{equation}
Finally, there exists $C_1\in(0,\infty)$ and a polynomial mapping $P_1$, such that for all $u\in H^1$, one has
\begin{equation}\label{eq:as1-2}
\begin{aligned}
\norm{F(u)}_{H^1}&\le C_1\norm{u}_{H^1}\left(1+\norm{u}_{L^2}^2\right)\\
|{\rm Im}(\langle \nabla u,\nabla F(u)\rangle)|&\le C_1\norm{\nabla u}_{L^2}^2
+P_1\left(\norm{u}_{L^2}^2\right).
\end{aligned}
\end{equation}
\end{assumption}
In the above and in the sequel, $\nabla u$ denotes the gradient of the mapping $u$.

Note that assuming that $V[u]$ is real-valued for all $u\in L^2$ implies that one has the equality ${\rm Im}(\langle u,F(u)\rangle)=0$.

The value of the real number $K_F$ appearing in the right-hand side of~\eqref{eq:KF} plays a crucial role in the convergence analysis below.

Let us recall the definition of the stochastic integral in the mild form \eqref{mild} and the associated It\^o isometry property. If for all $t\ge 0$, $\Psi(t)$ is a linear operator from $H^\sigma$ to $H^\sigma$, the stochastic integral $\displaystyle\int_0^T\Psi(t)\,\dd W^Q(t)$ is understood as
$\displaystyle\sum_{k\in\N}\gamma_k\int_0^T \Psi(t)e_k \,\dd \beta_k(t)$, and the It\^o isometry formula is given by
\[
\E\left[\norm{\int_{0}^{T}\Psi(t)\,\dd W^Q(t)}_{H^\sigma}^2\right]={\sum_{k\in\N}\int_0^T |\gamma_k|^2\norm{\Psi(t)e_k}_{H^\sigma}^2\,\dd t=}\int_0^T \norm{\Psi(t)Q^{\frac12}}_{\mathcal{L}_2^\sigma}^2\,\dd t.
\]
Under Assumption~\ref{as1}, the stochastic convolution $\displaystyle-\ii\int_{0}^{t}S(t-s)\,\dd W^Q(s)$
is thus well-defined and takes values in $H^1$. It solves the linear stochastic Schr\"odinger equation driven by additive noise
\[
\ii \, \mathrm{d}u(t)=\Delta u(t)\, \mathrm{d}t+\, \mathrm{d}W^Q(t),\quad u(0)=0.
\]

Most of the analysis can be performed when Assumption~\ref{as1} is satisfied, in particular we will prove
below that~\eqref{sse} admits a unique global solution, and that the splitting scheme has a convergence order $1/2$. To get convergence order $1$ of the proposed splitting integrator for the semilinear problem~\eqref{sse}, we need further assumptions.
\begin{assumption}\label{as2}
On top of Assumption~\eqref{as1}, let us assume that one has
$$
\norm{Q^{\frac12}}_{\mathcal{L}_2^2}^2=\sum_{k\in\N}|\gamma_k|^2\norm{e_k}_{H^2}^2<\infty.
$$
Furthermore, let us assume that the nonlinearity $F$ is twice differentiable, and there exists $C\in(0,\infty)$ such that for all $u,h,k\in L^2$, one has
\begin{equation}\label{eq:as2-1}
\begin{aligned}
\norm{F'(u).h}_{L^2}&\le C(1+\norm{u}_{L^2}^2)\norm{h}_{L^2}\\
\norm{F''(u).(h,k)}_{L^2}&\le C(1+\norm{u}_{L^2})\norm{h}_{L^2}\norm{k}_{L^2}.
\end{aligned}
\end{equation}
Finally, let us assume that there exists $C_2\in(0,\infty)$ and a polynomial mapping $P_2$, such that for all $u\in H^2$, one has
\begin{equation}\label{eq:as2-2}
\begin{aligned}
\norm{F(u)}_{H^2}&\le C_\sigma\norm{u}_{H^2}\left(1+\norm{u}_{L^2}^2\right)\\
|{\rm Im}(\langle \nabla^2 u,\nabla^2 F(u)\rangle)|&\le C_2\norm{\nabla^2 u}_{L^2}^2
+P_2\left(\norm{u}_{L^2}^2,\norm{\nabla u}_{L^2}^2\right).
\end{aligned}
\end{equation}
\end{assumption}
In the above and in the sequel, $\nabla^2 u$ denotes the Hessian matrix of the mapping $u$.

Next, we verify that the two examples of nonlinearities seen in the introduction, namely $F(u)=Vu$ and $F(u)=\left(V\star |u|^2\right)u$, verify these conditions.
First, the conditions in Assumption~\ref{as1}~or~\ref{as2} are satisfied in the case of a linear mapping $F(u)=Vu$,
where the external potential function $V\colon\mathbb{T}^d\to\R$
is a real-valued mapping of class $\mathcal{C}^\sigma$, with $\sigma=1$ (resp. $\sigma=2$) to satisfy Assumption~\ref{as1} (resp. Assumption~\ref{as2}). In that case, the mapping $F$ is globally Lipschitz continuous, and~\eqref{eq:KF} holds with $C_F=\norm{V}_{\mathcal{C}^0}$ and $K_F=0$.
Second, the conditions in Assumption~\ref{as1}~or~\ref{as2} also hold for the following class of nonlocal interaction cubic nonlinearities.  Note that $K_F\neq 0$ in this case.
\begin{proposition}\label{propo}
Let $\sigma\in\{1,2\}$ and let $V:\mathbb{T}^d\to\R$ be a real-valued mapping of class $\mathcal{C}^\sigma$. For every $u\in L^2$, set
\begin{equation*}
V[u]=V\star |u|^2=\int V(\cdot-x)|u(x)|^2\,\dd x,
\end{equation*}
where $\star$ denotes the convolution operator.

Then Assumption~\ref{as1} (resp. Assumption~\ref{as2}) is satisfied for the nonlinearity $F(u)=V[u]u=\left(V\star |u|^2\right)u$ when $\sigma=1$ (resp. when $\sigma=2$).
\end{proposition}
\begin{proof}
Observe that for any $u\in L^2$, the mapping $V[u]$ is of class $\mathcal{C}^\sigma$, with $\nabla^\sigma V[u]=\nabla^\sigma V \star |u|^2$ for $\sigma=1,2$. It thus follows that $\norm{V[u]}_{\mathcal{C}^\sigma}\le \norm{V}_{\mathcal{C}^\sigma}\norm{u}_{L^2}^2$ for all $u\in L^2$.

First, assume that $\sigma=1$. Let us check that~\eqref{eq:KF} holds. Let $u_1,u_2\in L^2$, then one has
\begin{align*}
\norm{F(u_2)-F(u_1)}_{L^2}&\le \norm{V[u_2](u_2-u_1)}_{L^2}+\norm{(V[u_2]-V[u_1])u_1}_{L^2}\\
&\le \norm{V[u_2]}_{\mathcal{C}^0}\norm{u_2-u_1}_{L^2}+\norm{V[u_2]-V[u_1]}_{\mathcal{C}^0}\norm{u_1}_{L^2}\\
&\le \norm{V}_{\mathcal{C}^0}\left(\norm{u_2}_{L^2}^2+\norm{u_1+u_2}_{L^2}\norm{u_1}_{L^2}\right)\norm{u_2-u_1}_{L^2}\\
&\le \frac32 \norm{V}_{\mathcal{C}^0}\left(\norm{u_1}_{L^2}^2+\norm{u_2}_{L^2}^2\right)\norm{u_2-u_1}_{L^2}.
\end{align*}
Thus~\eqref{eq:KF} holds with $C_F=0$ and $K_F=\frac32 \norm{V}_{\mathcal{C}^0}$. The conditions in~\eqref{eq:as1-2} follow from straightforward computations.

Second, assume that $\sigma=2$. The conditions in~\eqref{eq:as2-1} follow from writing, for all $u,h,k\in L^2$,
\begin{align*}
F'(u).h&=V[u]h+2\bigl(V\star {\rm Re}(\bar{u}h)\bigr)u\\
F''(u).(h,k)&=2\bigl(V\star {\rm Re}(\bar{u}h)\bigr)k+2\bigl(V\star {\rm Re}(\bar{k}u)\bigr)h+2\bigl(V\star {\rm Re}(\bar{h}k)\bigr)u.
\end{align*}
The conditions in~\eqref{eq:as2-2} follow from straightforward computations.

This concludes the proof of Proposition~\ref{propo}.
\end{proof}

Note that the conditions in Assumption~\ref{as1}~or~\ref{as2} are not satisfied in the standard cubic nonlinear Schr\"odinger case, where $V[u]=\pm|u|^2$,
or for other (non-trivial) power-law nonlinearities, or for the logarithmic nonlinearity considered in the recent preprint \cite{cui2021structurepreserving},
or for the stochastic Schr\"odinger--Poisson equation with a non-smooth potential $V$.

\begin{remark}
The result of Proposition~\ref{propo} remains valid if $D$ is an arbitrary smooth bounded domain and homogeneous Dirichlet boundary conditions are imposed, instead of periodic boundary conditions,
under appropriate assumptions on the potential $V$.
\end{remark}

To conclude this section, let us state a well-posedness result for the stochastic Schr\"odinger equation \eqref{sse}
in terms of mild solutions~\eqref{mild}, and several moment bound estimates. Note that additional bounds for
the exponential moments in $L^2$ of the exact solution are given in Section~\ref{sec-conv}.
\begin{proposition}\label{prop:exact}
Let Assumption~\ref{as1} be satisfied.

For any initial condition $u_0\in L^2$, there exists a unique mild solution $\bigl(u(t)\bigr)_{t\ge 0}$ of the stochastic
Schr\"odinger equation~\eqref{sse}, which satisfies~\eqref{mild} for all $t\ge 0$. In addition, for every $T\in(0,\infty)$, $\sigma\in\{0,1,2\}$, $u_0\in H^\sigma$, and $p\in[1,\infty)$, there exists $C_p(T,\alpha,Q,u_0)\in(0,\infty)$
such that one has a moment bound in $H^\sigma$
$$
\underset{0\le t\le T}\sup~\E[\norm{u(t)}_{H^\sigma}^{2p}]\le C_p(T,\alpha,Q,u_0),
$$
with $\sigma=1$, resp. $\sigma=2$, when Assumption~\ref{as1}, resp. Assumption~\ref{as2}, is satisfied.

Finally one has the following temporal regularity estimate: for all $t_1,t_2\in[0,T]$,
$$
\E\left[\norm{u(t_2)-u(t_1)}_{L^2}^{2p}\right]\le C_p(T,\alpha,Q,u_0)|t_2-t_1|^p.
$$
\end{proposition}
The proof uses standard arguments and is postponed to the appendix.

\section{Splitting scheme}\label{sec-split}
In this section we define a splitting integrator for the stochastic Schr\"odinger equation~\eqref{sse} and show some
geometric properties of this time integrator.
The main idea of splitting schemes is to decompose the original problem, equation~\eqref{sse} in our case,
into subsystems that can be solved exactly (or efficiently numerically).
Splitting schemes are widely used for time discretization of deterministic cubic Schr\"odinger equations,
see, e.g. the key early reference \cite{ht73}.

The definition of the splitting scheme studied in this article relies on the flow associated with
the differential equation $\ii \dot{u}=F(u)=V[u]u$. For all $u\in L^2$ and $t\in\R$, define
\begin{equation}
\Phi_{t}(u)=\e^{-\ii tV[u]}u.
\end{equation}
Since $V[u]$ is real-valued by Assumption~\ref{as1}, one has $|\Phi_{t}(u)|=|u|$ for all $t\ge 0$, which gives $V[\Phi_{t}(u)]=V[u]$
using Assumption~\ref{as1}~or~\ref{as2}. It is then straightforward to check that
$\bigl(\Phi_t\bigr)_{t\in\R}$ is the flow associated with the differential
equation $\ii \dot{u}=F(u)$. Indeed, for all $u\in L^2$ and all $t\in\R$, one has
\[
\ii\frac{\dd}{\dd t}\Phi_{t}(u)=V[u]\Phi_t(u)=F\left(\Phi_t(u)\right).
\]

Observe that the flow of the above ODE preserves the $L^2$-norm: one has
$\norm{\Phi_t(u)}_{L^2}=\norm{u}_{L^2}$ for all $t\ge 0$ and all $u\in L^2$.

The splitting scheme for the stochastic Schr\"odinger equation \eqref{sse} considered in this article is then
defined by the explicit recursion
\begin{equation}\label{S1}
u_{n+1}=S(\tau)\left(\Phi_\tau(u_n)-\ii\alpha\delta W_n^Q\right),
\end{equation}
where $\tau$ denotes the time-step size, and $\delta W_n^Q=W^Q((n+1)\tau)-W^Q(n\tau)$ are Wiener increments.
Recall that $S(\tau)=\e^{-\ii \tau\Delta}$ is defined after equation~\eqref{mild}. Without loss of generality, it is assumed that $\tau\in(0,1)$. The scheme is obtained using a splitting strategy: at each time step, first one may write
$\tilde{u}_n=\Phi_{\tau}(u_n)$, {\it i.\,e.} the equation $\ii \dot{u}=F(u)$ with initial condition $u_n$ is solved exactly,
second one has $u_{n+1}=S(\tau)\tilde{u}_n-\ii\alpha S(\tau)\delta W_n^Q$, which comes from applying
an exponential Euler scheme to the stochastic differential equation $\ii\dd u=\Delta u\,\dd t+\,\alpha\dd W^Q(t)$.
Observe that bounds for the exponential moments in $L^2$ of the numerical solution are given in Section~\ref{sec-conv}.

\begin{remark}
Alternatively, solving exactly the stochastic differential equation $\ii\dd u=\Delta u\,\dd t+\,\alpha\dd W^Q(t)$
yields the following numerical scheme for the SPDE \eqref{sse}
\begin{equation}\label{S2}
u_{n+1}=S(\tau)\Phi_\tau(u)-\ii\alpha\int_{n\tau}^{(n+1)\tau}S((n+1)\tau-t)\,\dd W^Q(t).
\end{equation}
Generalizing the results obtained below to this numerical scheme is straightforward and thus omitted in the sequel.
\end{remark}

The error analysis for the splitting scheme \eqref{S1} presented in the next section will make use of the following additional assumption.
\begin{assumption}\label{as3}
There exists $C\in(0,\infty)$ such that for all $t\in[0,1]$ and $u\in L^2$ one has
\[
\norm{\Phi_t(u)-u}_{L^2}\le C|t|\left(1+\norm{u}_{L^2}^3\right).
\]
\end{assumption}
Note that Assumption~\ref{as3} is satisfied for the two examples of nonlinearities described in Section~\ref{notation}.
Indeed, one obtains $\norm{\Phi_t(u)-u}_{L^2}\le t\norm{V[u]}_{\mathcal{C}^0}\norm{u}_{L^2}$.
For the case of an external potential ($V[u]=V$), one has $\norm{V[u]}_{\mathcal{C}^0}=\norm{V}_{\mathcal{C}^0}$.  For the case of a nonlocal interaction ($V[u]=V\star|u|^2$), one has $\norm{V[u]}_{\mathcal{C}^0}\le \norm{V}_{\mathcal{C}^0}\norm{u}_{L^2}^2$.

We now present some geometric properties of the splitting scheme~\eqref{S1}.

\subsection{Trace formula for the mass}
It is well known that, under periodic boundary conditions for instance, the mass,
or $L^2$-norm or density
$$
M(u):=\norm{u}^2_{L^2}=\int |u|^2\,\dd x
$$
of the deterministic Schr\"odinger equation $\ii \frac{\partial u}{\partial t}-\Delta u-V[u]u=0$,
where $V[u]=V$ (external potential) or $V[u]=V\star|u|^2$ (nonlocal interaction) or $V[u]=|u|^2$ (cubic nonlinearity),
is a conserved quantity.
In the stochastic case under consideration, one immediately gets a trace formula
for the mass of the exact solution of equation \eqref{sse} as well as for the numerical
solution given by the splitting scheme \eqref{S1}.

\begin{proposition}\label{prop:mass}
Consider the stochastic Schr\"odinger equation \eqref{sse} with a trace class covariance operator $Q$
and an initial value satisfying $\E[M(u_0)]<\infty$. We assume that the nonlinearity in \eqref{sse}
is such that $F(u)=V[u]u$, where $V[u]$ is real-valued and a function of the modulus $|u|$. Furthermore,
we assume that an exact global solution exists. Finally, we assume\footnote{This is the case for instance when one considers an external potential, a nonlocal interaction,
a cubic or power-law nonlinearity.} that the differential equation,
$\ii \dot{u}=F(u)=V[u]u$, in the splitting scheme can be solved exactly.

Then, the exact solution \eqref{mild} satisfies a trace formula for the mass:
\begin{align*}
\E\left[M(u(t))\right]=\E\left[\norm{u(t)}^2_{L^2}\right]=
\E\left[M(u_0)\right]+t\alpha^2\text{Tr}(Q)\quad\text{for all times}\quad t.
\end{align*}
Furthermore, the numerical solution given by the splitting scheme \eqref{S1} to the nonlinear stochastic Schr\"odinger equation \eqref{sse} satisfies
the exact same trace formula for the mass:
\begin{align*}
\E\left[M(u_n)\right]=\E\left[M(u_0)\right]+t_n\alpha^2\text{Tr}(Q)\quad\text{for all times}\quad t_n=n\tau.
\end{align*}
\end{proposition}
Observe that the above result for the exact solution is already available in the literature in different settings,
for instance in \cite{MR1954077,MR3771721}. However, to the best of our knowledge, the result for the numerical solution is one of the first results in the literature on a longtime qualitative behavior of explicit numerical solutions to nonlinear SPDEs driven by It\^o noise. Such a longtime behavior is not satisfied for classical time integrators like the (semi-implicit) Euler--Maruyama schemes, see the numerical experiments in Section~\ref{sec-num}. Trace formulas for numerical schemes applied to stochastic linear Schr\"odinger, wave, and Maxwell
equations driven by additive noise have been obtained in~\cite{MR3771721,MR3033008,MR4077824}.
For implicit schemes applied to nonlinear stochastic wave equations, we refer to \cite{MR4362830}.

\begin{proof}
We apply It\^o's formula to the mass $M(u(t))$ and get
\begin{equation}\label{eq:trace}
\begin{aligned}
M(u(t))&=M(u(0))+\int_0^t\langle M'(u(s)),\,-\ii\alpha\dd W(s)\rangle \\
&\quad+\int_0^t\langle M'(u(s)),-\ii\Delta u(s)-\ii V[u]u(s) \rangle\,\dd s\\
&\quad+\int_0^t\frac12\alpha^2\text{Tr}\left[M''(u(s))\left(Q^{1/2}\right)\left(Q^{1/2}\right)^*\right]\,\dd s.
\end{aligned}
\end{equation}

An integration by parts and the hypothesis on the potential $V$ show that the third term on the right-hand side is zero.
Taking expectation now gives
\begin{align*}
\E\left[M(u(t))\right]=\E\left[M(u(0))\right]+t\alpha^2\text{Tr}(Q)
\end{align*}
which concludes the proof of the trace formula for the mass of the exact solution.

We next show that the above trace formula is also satisfied for the numerical
solution given by the splitting integrator \eqref{S1}. Using the definition of the numerical scheme \eqref{S1},
properties of the Wiener increments $\delta W_n^Q$, as well as the isometry property of $S(\tau)$, one gets
\begin{align*}
\E\left[ M(u_{n+1}) \right]&=\E\left[ \norm{S(\tau)\Phi_\tau(u_n)}^2_{L^2} \right]+\alpha^2\E\left[ \norm{\delta W_n^Q}_{L^2}^2 \right]\\
&=\E\left[ \norm{\Phi_\tau(u_n)}_{L^2}^2 \right]+\tau\alpha^2\text{Tr}(Q).
\end{align*}
The isometry property of the flow $\Phi_\tau$ yields
\begin{align*}
\E\left[ M(u_{n+1}) \right]=\E\left[ M(u_{n}) \right]+\tau\alpha^2\text{Tr}(Q)
\end{align*}
and a recursion completes the proof of the proposition.
\end{proof}

\begin{remark}
The same trace formula for the mass holds for the numerical solution given by the time integrator \eqref{S2}.
Indeed, using the definition of the numerical scheme \eqref{S2}, It\^o's isometry,
as well as the isometry property of the operator $S(\tau)$ and of the flow $\Phi_\tau$, one also gets
\begin{align*}
\E\left[ M(u_{n+1}) \right]=\E\left[ \norm{\Phi_\tau(u_n)}_{L^2}^2 \right]+\tau\alpha^2\text{Tr}(Q)
=\E\left[ M(u_{n}) \right]+\tau\alpha^2\text{Tr}(Q).
\end{align*}
\end{remark}

\begin{remark}
It may also be possible to study the longtime behavior of the exact and numerical solutions along
the expected value of the Hamiltonian of \eqref{sse} with $\alpha\neq0$. However, in general, the drift in the expected
Hamiltonian will depend on the solution $u$, see for example \cite[Equation~$(11)$]{Numericalsimulationfocusingstochastic}
for the cubic case. In particular, the evolution of this quantity will not be linear in time.
Such a trace formula for the energy will thus unfortunately not be as simple
as the one for the mass. Very recent studies have been carried on for (mainly) the Crank--Nicolson scheme in the
preprint \cite{mry20}. In particular, it is observed that
this numerical scheme does not verify an exact trace formula for the mass,
see also the numerical experiments below.
We leave the question of investigating such trace formula for the Hamiltonian of the splitting scheme for future work.
\end{remark}

\subsection{Stochastic symplecticity}
Symplectic schemes are known to have excellent longtime properties when applied
to Hamiltonian (partial) differential equations,
see for instance \cite{MR2132573,MR2840298,MR2193969,MR2425152,
MR2995211,MR3401810,MR3712186} and references therein.
These particular integrators have thus naturally
come into the realm of stochastic (partial) differential equations,
see for example \cite{MR2491434,MR2926251,MR3218332,MR1897950,MR3542010,MR3986278,cchs19}
and references therein.

The stochastic Schr\"odinger equation can be interpreted as a canonical
infinite-dimensional Hamiltonian system, see \cite{MR3986278}.
The next result shows that the exact flow of the SPDE~\eqref{sse}
as well as the proposed splitting scheme \eqref{S1} are stochastic symplectic.

\begin{proposition}\label{prop:symp}
Consider the stochastic Schr\"odinger equation \eqref{sse}
and assume that a global solution exists.
Under the same assumptions as in Proposition~\ref{prop:mass},
the exact flow of this SPDE
is stochastic symplectic in the sense that it preserves the symplectic form
$$
\bar\omega(t)=\int_{\mathbb{T}^d}\dd p\wedge \dd q\,\dd x\quad\text{a.s.},
$$
where the overbar on $\omega$ is a reminder that the two-form $\dd p\wedge \dd q$
(with differentials made with respect to the initial value) is integrated over the torus.
Here, $p$ and $q$ denote the real and imaginary parts of $u$.

Furthermore, the splitting scheme \eqref{S1} applied to the
stochastic Schr\"odinger equation \eqref{sse} is stochastic
symplectic in the sense that it possesses the discrete symplectic structure:
$$
\bar\omega^{n+1}=\bar\omega^{n}\quad\text{a.s.},
$$
for the symplectic form $\bar\omega^{n}:=\displaystyle\int_{\mathbb{T}^d}\dd p_{n}\wedge \dd q_{n}\,\dd x$,
where $p_n$, resp. $q_n$ denoting the real and imaginary parts of $u_n$,
and $\dd$ denotes differentials in the phase space.
\end{proposition}
\begin{proof}
The symplecticity of the phase flow of the stochastic Schr\"odinger equation \eqref{sse}
can be shown using similar arguments as in \cite[Theorem 3.1]{MR3986278} for
a stochastic cubic Schr\"odinger equation with quadratic potential, see also \cite{MR3542010}.

In order to show that the numerical solution is stochastic symplectic as well,
we use the same argument as in the proof of \cite[Prop. 4.3]{cchs19}.
Taking the differential of the numerical solution yields
\begin{align*}
\dd u_{n+1}=\dd\left( S(\tau)\left( \Phi_\tau(u_n)-\ii\alpha\delta W_n^Q\right) \right)=
\dd\left(S(\tau)\Phi_\tau(u_n)\right)=\dd u_n,
\end{align*}
where in the last equality we have used the fact that the composition of exact flows is symplectic.
This concludes the proof.
\end{proof}
\begin{remark}
The exact same proof shows that the splitting scheme \eqref{S2} possesses a discrete symplectic structure.
\end{remark}

\section{Convergence results}\label{sec-conv}
In this section, we study various types of convergence (strong, in probability and almost-surely)
of the splitting scheme \eqref{S1} when applied to the stochastic Schr\"odinger equation \eqref{sse}.
In order to do this, we first show bounds for the exponential moments in the $L^2$ norm of the
exact and numerical solutions as well as two auxiliary results.
The proofs of these results are given in Section~\ref{sec-proofs} for the reader's convenience.
These proofs could also be obtained using tools from~\cite{MR4079431}.

\begin{theorem}\label{thm-expm}
Let us apply the splitting scheme \eqref{S1} to the stochastic Schr\"odinger equation \eqref{sse}
with a trace class covariance operator $Q$ and deterministic initial value $u_0\in L^2$.
Assume that the nonlinearity in \eqref{sse} satisfies
$F(u)=V[u]u$, where $V[u]=V[\bar{u}]$ is real-valued and that Assumption~\ref{as1} holds. One then has the following bounds for the exponential moments:
there exists $\kappa>0$ such that
if $\mu \alpha^2T<\frac{\kappa}{{\rm Tr}(Q)}$, then one has:
$$
\underset{0\le t\le T}\sup~\E\left[\exp\left(\mu\norm{u(t)}^2_{L^2}\right)\right]\leq
C(\mu,T,\alpha,Q,u_0)<\infty
$$
for the exact solution and there exists $\tau^\star>0$ such that
$$
\underset{\tau\in(0,\tau^\star)}\sup~\underset{0\le n\tau\le T}\sup~\E\left[\exp\left(\mu\norm{u_n}^2_{L^2}\right)\right]
\leq C(\mu,T,\alpha,Q,u_0)<\infty
$$
for the numerical solution.
\end{theorem}
In the proof of Theorem~\ref{thm-expm}, the lower bound $\kappa\geq\frac{e^{-1}}{2}$ is obtained, note that it does not depend on the nonlinearity. Furthermore, observe that the condition $\mu \alpha^2T<\frac{\kappa}{{\rm Tr}(Q)}$ gets more restrictive when $\alpha$ and $T$ increase. In addition, the proof of Theorem~\ref{thm-expm} reveals that one may choose any $\tau^\star\in\bigl(0,p(p-1)\bigr)$ where $p=\frac{e^{-1}}{\mu 2\alpha^2 T{\textup{Tr}}(Q)}>1$.



It is immediate to deduce the following moment estimates for the exact and numerical solutions from Theorem~\ref{thm-expm}.
\begin{corollary}\label{cor:momentsL2}
Under the assumptions of the previous theorem, for any $p\in[1,\infty)$ and $T\in(0,\infty)$, one has the following moment estimates for the $L^2$ norm of the exact and numerical solutions: for any $u_0\in L^2$, there exists $C_p(T,\alpha,Q,u_0)\in(0,\infty)$ such that
$$
\underset{0\le t\le T}\sup~\E[\norm{u(t)}_{L^2}^{2p}]\le C_p(T,\alpha,Q,u_0)
$$
and
$$
\underset{\tau\in(0,1)}\sup~\underset{0\le n\tau\le T}\sup~\E[\norm{u_n}_{L^2}^{2p}]\le C_p(T,\alpha,Q,u_0).
$$
\end{corollary}

In order to show the main convergence result of this article, we will make use of the
following two propositions. Each one of these propositions are used to show strong convergence order $1/2$,
resp. $1$, of the numerical solution given by the splitting scheme.

\begin{proposition}\label{prop:conv1/2}
Consider the time discretization of the stochastic Schr\"odinger equation \eqref{sse} by the splitting scheme \eqref{S1}. Let Assumptions~\ref{as1} and~\ref{as3} be satisfied. Assume that $u_0\in H^1$.

Let $T\in(0,\infty)$. For every $q\in[1,\infty)$, there exists $\tau^\star>0$ and $C_q(T,u_0)\in(0,\infty)$ (which depends on $F$, $Q$ and on $\alpha$), such that for every $\tau\in(0,\tau^\star)$, one has
\[
\underset{0\le n\tau\le T}\sup~\E\left[\exp(-qK_FS_n)\norm{u_n-u(t_n)}_{L^2}^{q}\right]\le C_q(T,u_0)\tau^{\frac{q}{2}},
\]
where $S_n=\displaystyle\tau\sum_{k=0}^{n-1}\left(\norm{u(k\tau)}_{L^2}^2+\norm{u_k}_{L^2}^2\right)$, and $K_F$ is given in~\eqref{eq:KF} (see Assumption~\ref{as1}).
\end{proposition}

\begin{proposition}\label{prop:conv1}
Consider the time discretization of the stochastic Schr\"odinger equation \eqref{sse} by the splitting scheme \eqref{S1}. Let Assumptions~\ref{as2} and~\ref{as3} be satisfied. Assume that $u_0\in H^2$.

Let $T\in(0,\infty)$. For every $q\in[1,\infty)$, there exists $\tau^\star>0$ and $C_q(T,u_0)\in(0,\infty)$ (which depends on $F$, $Q$ and on $\alpha$), such that for every $\tau\in(0,\tau^\star)$, one has
\[
\underset{0\le n\tau\le T}\sup~\E\left[\exp(-qK_FS_n)\norm{u_n-u(t_n)}_{L^2}^{q}\right]\le C_q(T,u_0)\tau^{q},
\]
where $S_n=\displaystyle\tau\sum_{k=0}^{n-1}\left(\norm{u(k\tau)}_{L^2}^2+\norm{u_k}_{L^2}^2\right)$, and $K_F$ is given in~\eqref{eq:KF} (see Assumption~\ref{as1}).
\end{proposition}


The proofs of the technical results, Theorem~\ref{thm-expm} and Propositions~\ref{prop:conv1/2} and~\ref{prop:conv1}, are postponed to Section~\ref{sec-proofs}.

We are now in position to state the main convergence result of this article.
\begin{theorem}\label{th:conv}
Let $u(t)$ denote the exact solution to the stochastic Schr\"odinger equation~\eqref{sse} and $u_n$ the numerical solution given by the splitting scheme~\eqref{S1}. Let Assumption~\ref{as3} be satisfied. Let also $\sigma=1$, resp. $\sigma=2$,
if Assumption~\ref{as1}, resp. Assumption~\ref{as2}, is satisfied. Assume that $u_0\in H^\sigma$.

Recall the notation
$S_n=\displaystyle\tau\sum_{k=0}^{n-1}\left(\norm{u(k\tau)}_{L^2}^2+\norm{u_k}_{L^2}^2\right)$.
Let $T\in(0,\infty)$. Assume that $\overline{\mu}\in(0,\infty)$ and $\tau_0\in(0,\tau^\star)$
are chosen such that
\begin{equation}\label{eq:condition_mu-tau0}
\underset{\tau\in(0,\tau_0)}\sup~\underset{0\le n\tau\le T}\sup~\E\left[\exp(\overline{\mu}S_n)\right]=C(T,u_0,\alpha,Q,\tau_0,\overline{\mu})<\infty.
\end{equation}
Then, for all $r\in(0,\infty)$ and all $\mu\in(0,\overline{\mu})$, there exists $C_r(T,u_0,\alpha,Q,\tau_0,\mu)<\infty$ such that for all $\tau\in(0,\tau_0)$ one has
\begin{equation}\label{eq:conv_mu-r}
\underset{0\le n\tau\le T}\sup~\left(\E\left[\norm{u_n-u(t_n)}_{L^2}^r\right]\right)^{\frac{1}{r}}\le
C_r(T,u_0,\alpha,Q,\tau_0,\mu)\tau^{\frac{\sigma}{2}\min(1,\frac{\mu}{rK_F})}.
\end{equation}
\end{theorem}

The positive parameter $r$ in Theorem~\ref{th:conv} can be chosen arbitrarily close to $0$. However, one needs to be careful when using values of $r$ which are smaller than $1$: indeed the mapping $(X,Y)\mapsto \left(\mathbb{E}[|X-Y|^r]\right)^{\frac{1}{r}}$ is not a distance. Being able to choose arbitrarily small positive values of $r$ is important in the analysis,
in order to prove Corollary~\ref{cor:pas} below.

As a consequence of Theorem~\ref{th:conv} above, the convergence is polynomial in $L^r(\Omega)$, for all $r\in[1,\infty)$.
The rate of convergence of the splitting scheme depends on $r$ in~\eqref{eq:conv_mu-r}, and vanishes when $r\to \infty$. Note that for sufficiently small $r>0$, one has $\min(1,\frac{\mu}{rK_F})=1$, thus the convergence rate of the splitting scheme is $\frac{\sigma}{2}$ when $r$ is sufficiently small. Observe also that a sufficient condition for condition \eqref{eq:condition_mu-tau0} to
be verified is that
\[
\overline{\mu}<\frac{\kappa}{\alpha^2T^2{\rm Tr}(Q)},
\]
where $\kappa>0$ is some positive constant (see Theorem~\ref{thm-expm} above and Remark~\ref{rem:expm} below). Thus the value of $\min(1,\frac{\mu}{rK_F})$ depends on the quantity $\alpha^2T^2K_F$ (considering that ${\rm Tr}(Q)$ is fixed and that the size of the noise is given by $\alpha$). The larger this quantity, the more restrictive the condition to have $\min(1,\frac{\mu}{rK_F})=1$ becomes.

In the external potential case $V[u]=Vu$, one has $K_F=0$, thus there is no restrictions and the order of convergence is $\frac{\sigma}{2}$ in $L^r(\Omega)$ for all $r\in[1,\infty)$.

\begin{remark}\label{rem:expm}
Owing to Theorem~\ref{thm-expm} concerning exponential moments of the exact and numerical solutions, the set of parameters $\overline\mu,\tau_0$ such that~\eqref{eq:condition_mu-tau0} holds is non-empty. Indeed,
recalling that $S_n=\displaystyle\tau\sum_{k=0}^{n-1}\left(\norm{u(k\tau)}_{L^2}^2+\norm{u_k}_{L^2}^2\right)$ and
using Cauchy--Schwarz inequality, one has
\begin{align*}
\E\left[\exp(\mu S_n)\right]&=\E\left[\prod_{k=0}^{n-1}
\left(\exp(\mu\tau\norm{u(t_k)}_{L^2}^2)\exp(\mu\tau\norm{u_k}_{L^2}^2)\right)\right]\\
&\le \prod_{k=0}^{n-1}\left(\E\left[\exp(2n\tau\mu\norm{u(t_k)}_{L^2}^2)\right]\E\left[\exp(2n\tau\mu\norm{u_k}_{L^2}^2)\right]\right)^{\frac{1}{2n}}\\
&\le \underset{0\le k\le n}\sup~\E\left[\exp(2T\mu\norm{u(t_k)}_{L^2}^2)\right]
\underset{0\le k\le n}\sup~\E\left[\exp(2T\mu\norm{u_k}_{L^2}^2)\right]\\
&\le C(\mu,T,\alpha,Q,u_0),
\end{align*}
if $\mu<\frac{\kappa}{\alpha^2T^2{\rm Tr}(Q)}$ and $\tau<\tau^\star$, where $\kappa$ and $\tau^\star$ are given in Theorem~\ref{thm-expm}. The value of $\mu$ obtained by the argument above (as well as the values of $\kappa=\frac{\e^{-1}}{2}$ and $\tau^\star$) may not be optimal.
\end{remark}

\begin{proof}[Proof of Theorem~\ref{th:conv}]
Set $e_n=\norm{u_n-u(t_n)}_{L^2}$. For every $R\in(0,\infty)$, let $\chi_{n,R}=1_{S_n\le R}$. Then
\[
\E[e_n^r]=\E[e_n^r\chi_{n,R}]+\E[e_n^r(1-\chi_{n,R})].
\]
For a given $\mu\in(0,\overline{\mu})$, let $p\in(1,\infty)$ such that $\mu=(1-\frac{1}{p})\overline{\mu}$.

On the one hand, applying the Cauchy--Schwarz and Markov
inequalities yields
\begin{align*}
\E\left[e_n^r(1-\chi_{n,R})\right]&\le \left(\E\left[e_n^{rp}\right]\right)^{\frac{1}{p}}
\left(\E\left[1-\chi_{n,R}\right]\right)^{1-\frac1p}\le
\left(\E[e_n^{rp}]\right)^{\frac{1}{p}} \mathbb{P}(S_n>R)^{1-\frac1p} \\
&\le \left(\E[e_n^{rp}]\right)^{\frac{1}{p}} \mathbb{P}\left(\exp(\overline{\mu}S_n)>\exp(\overline{\mu}R)\right)^{1-\frac1p}\\
&\le \left(\E[e_n^{rp}]\right)^{\frac{1}{p}}
\left(\frac{\E[\exp(\overline{\mu}S_n)]}{\exp(\overline{\mu}R)}\right)^{1-\frac1p}.
\end{align*}
Using moment bounds for the exact and the numerical solution (Corollary~\ref{cor:momentsL2}) and the exponential moment estimate~\eqref{eq:condition_mu-tau0} for $S_n$, then yield
(for a constant $C$ that does not depend on $R$)
\[
\E\left[e_n^r(1-\chi_{n,R})\right]\le C\e^{-\mu R}.
\]
On the other hand, let $q=pr$ for $p$ introduced above. Applying the Cauchy--Schwarz inequality yields
\begin{align*}
\E[e_n^r\chi_{n,R}]=\E[e_n^r\e^{-rK_FS_n}\e^{rK_FS_n}\chi_{n,R}]\le
\left(\E\left[e_n^q\e^{-qK_FS_n}\right]\right)^{\frac1p}\left(\E\left[\e^{\frac{rK_Fp}{p-1}S_n}\chi_{n,R}\right]\right)^{1-\frac1p}.
\end{align*}
Using Proposition~\ref{prop:conv1/2} with $\sigma=1$, or Proposition~\ref{prop:conv1} with $\sigma=2$, and the relation $q=pr$, for the first factor one has
\[
\left(\E\left[e_n^q\e^{-qK_FS_n}\right]\right)^{\frac1p}\le C\tau^{\frac{r\sigma}{2}}.
\]
For the second factor, using the exponential moment estimates and the upper bound $S_n\le R$ when $\chi_{n,R}\neq 0$, one obtains
\begin{align*}
\left(\E\left[ \e^{\frac{rK_Fp}{p-1}S_n}\chi_{n,R}\right]\right)^{1-\frac1p}&\le
\left(\E\left[\e^{\overline{\mu}S_n}\right]\exp(\max(0,\frac{rK_Fp}{p-1}-\overline{\mu})R)\right)^{1-\frac1p}\\
&\le C \exp\left(\max(0,rK_F-\mu)R\right),
\end{align*}
using the identity $\mu=(1-\frac{1}{p})\overline{\mu}$.

Finally, for all $R\in(0,\infty)$, one has
\[
\E[e_n^r]\le C\left(\tau^{\frac{r\sigma}{2}}\exp\bigl(\max(0,rK_F-\mu)R\bigr)+\exp(-\mu R)\right).
\]
It remains to optimize the choice of $R$ in terms of $\tau$. If $rK_F\le \mu$, there is no condition and passing to the limit $R\to\infty$ yields $\E[e_n^r]\le C\tau^{\frac{r\sigma}{2}}$. If $rK_F>\mu$, the right-hand side is minimized when
$\tau^{\frac{r\sigma}{2}}\e^{rK_F R}=1$, {\it i.\,e.} $\e^{-R}=\tau^{\frac{r\sigma}{2K_F}}$ and one obtains
\[
\E[e_n^r]\le C\tau^{\frac{r\sigma\mu}{2K_F}}.
\]
This concludes the proof of the theorem.
\end{proof}

To conclude this section, let us state results concerning convergence in probability,
with order of convergence equal to $\frac{\sigma}{2}$, and almost sure convergence with order
of convergence $\frac{\sigma}{2}-\varepsilon$ for all $\varepsilon\in(0,\frac12)$, with $\sigma\in\{1,2\}$.

\begin{corollary}\label{cor:pas}
Consider the stochastic Schr\"odinger equation \eqref{sse} on the time interval $[0,T]$ with solution denoted by $u(t)$.
Let $u_n$ be the numerical solution given by the splitting scheme \eqref{S1} with time-step size $\tau$.
Under the assumptions of Theorem~\ref{th:conv}, one has convergence in probability of order $\frac{\sigma}{2}$
\[
\underset{C\to\infty}\lim~\mathbb{P}\left(\norm{u_N-u(T)}_{L^2}\ge C\tau^{\frac{\sigma}{2}}\right)=0,
\]
where $T=N\tau$.

Moreover, consider the sequence of time-step sizes given by $\tau_M=\frac{T}{2^M}$, $M\in\N$.
Then, for every $\varepsilon\in(0,\frac{\sigma}{2})$, there exists an almost surely finite random variable
$C_\varepsilon$, such that for all $M\in\N$ one has
\[
\norm{u_{2^M}-u(T)}_{L^2}\le C_{\varepsilon}\left(\frac{T}{2^M}\right)^{\frac{\sigma}{2}-\varepsilon}.
\]
\end{corollary}

\begin{proof}
Let $r$ be chosen sufficiently small, such that applying Theorem~\ref{th:conv} yields
\[
\E\left[\norm{u_N-u(T)}_{L^2}^r\right]\le C(r,T)\tau^{\frac{r\sigma}{2}}.
\]
Then the convergence in probability result is a straightforward consequence of Markov's inequality:
\begin{align*}
\mathbb{P}\left(\norm{u_N-u(T)}_{L^2}\ge C\tau^{\frac{\sigma}{2}}\right)&=
\mathbb{P}\left(\norm{u_N-u(T)}_{L^2}^r\ge C^r\tau^{\frac{r\sigma}{2}}\right)\\
&\le
\frac{\E\left[\norm{u_N-u(T)}_{L^2}^r\right]}{C^r\tau^{\frac{r\sigma}{2}}}=\frac{C(r,T)}{C^r}\underset{C\to\infty}\to 0.
\end{align*}
To get the almost sure convergence result, it suffices to observe that (again by applying Theorem~\ref{th:conv})
\[
\sum_{m=0}^{\infty}\frac{\E\left[\norm{u_{2^{m}}-u(T)}_{L^2}^r\right]}{\tau_{m}^{r(\frac{\sigma}{2}-\varepsilon)}}<\infty,
\]
thus $\frac{\norm{u_{2^M}-u(T)}_{L^2}^r}{\tau_M^{r(\frac{\sigma}{2}-\varepsilon)}}\underset{M\to\infty}\to 0$ almost surely.
\end{proof}

We do not know whether the rates of convergence in Theorem~\ref{th:conv} are optimal, indeed our arguments have limitations due to the use of exponential moment bounds from Theorem~\ref{thm-expm} which may not be optimal. In particular, the size $\alpha$ of the noise and the length $T$ of the time interval have some influence, which may be unexpected.

Like in Theorem~\ref{th:conv}, in the proof of Corollary~\ref{cor:pas} it is possible to choose arbitrarily small positive parameters $r$, with the same care in the interpretation when $r\le 1$. This is an important result and the reason why we are able to obtain orders of convergence $1/2$ or $1$ (depending on the value of the regularity parameter $\sigma$) for the convergence in probability and the almost sure convergence, even if in average or mean-square sense one may have lower orders of convergence.


\section{Proofs of technical results}\label{sec-proofs}

This section is devoted to giving the proofs to Theorem~\ref{thm-expm} and Propositions~\ref{prop:conv1/2}~and~\ref{prop:conv1}.

To simplify notation, we let $Q_\alpha=\alpha^2 Q$, where we recall that $Q$ is the covariance operator of the noise in the SPDE~\eqref{sse}.
\subsection{Proof of Theorem~\ref{thm-expm}}
We start with the proof of Theorem~\ref{thm-expm}.

\begin{proof}
Set $\lambda=\frac{1}{2T{\rm Tr}(Q_\alpha)}$ and define the stochastic process
$X(t)=\e^{-t/T}\norm{u(t)}_{L^2}^{2}$.
An application of It\^o's formula gives
\begin{align*}
\mathrm{d}\left(\e^{\lambda X(t)}\right)&=
\e^{\lambda X(t)}\left(-\lambda/T X(t)\,\mathrm{d}t+\lambda\e^{-t/T}{\rm Tr}(Q_\alpha)\,\mathrm{d}t
+\frac{\lambda^2}2\mathrm{d}\langle X\rangle_t\right)\\
&\quad+2\lambda\e^{\lambda X(t)}\e^{-t/T}\langle u(t),\mathrm{d}W^{Q_\alpha}(t)\rangle,
\end{align*}
where the quadratic variation $\langle X\rangle_t$ satisfies
$$
\mathrm{d}\langle X\rangle_t\leq\e^{-2t/T}4{\rm Tr}(Q_\alpha)\norm{u(t)}_{L^2}^2\,\mathrm{d}t
\leq4{\rm Tr}(Q_\alpha)\e^{-t/T}X(t)\,\mathrm{d}t.
$$
Taking expectation in the first equation above and observing that $X(t)\geq0$ a.s, one gets
\begin{align*}
\frac{\mathrm{d}\E[\e^{\lambda X(t)}]}{\mathrm{d}t}&\leq\lambda{\rm Tr}(Q_\alpha)\E[\e^{\lambda X(t)}]
+\E[\e^{\lambda X(t)}\left(2\lambda^2{\rm Tr}(Q_\alpha)-\lambda/T\right)X(t)]\\
&\leq\lambda{\rm Tr}(Q_\alpha)\E[\e^{\lambda X(t)}]
\end{align*}
by definition of $\lambda$.

By definition of the stochastic process $X(t)$, the above reads
$$
\frac{\mathrm{d}\E\left[\exp\left(\lambda\e^{-t/T}\norm{u(t)}_{L^2}^2 \right)\right]}{\mathrm{d}t}
\leq\lambda{\rm Tr}(Q_\alpha)\E\left[\exp\left(\lambda\e^{-t/T}\norm{u(t)}_{L^2}^2 \right)\right]
$$
and applying Gronwall's lemma provides the following estimate
$$
\E\left[\exp\left(\lambda\e^{-t/T}\norm{u(t)}_{L^2}^2\right)\right]\leq
\exp\left(\lambda\norm{u_0}_{L^2}^2\right)\e^{\lambda{\rm Tr}(Q_\alpha)t},
$$
Finally, let $\mu\le \frac{\e^{-1}}{2T{\rm Tr}(Q_\alpha)}=\e^{-1}\lambda$. Then for all $t\in[0,T]$,
\begin{align*}
\E\left[\exp\left(\mu\norm{u(t)}_{L^2}^2\right)\right]&\le \E\left[\exp\left(\lambda\e^{-t/T}\norm{u(t)}_{L^2}^2\right)\right]\\
&\le \exp\left(\lambda\norm{u_0}_{L^2}^2\right)\e^{\lambda{\rm Tr}(Q_\alpha)T},
\end{align*}
where we recall that $\lambda=\frac{1}{2T{\rm Tr}(Q_\alpha)}=\frac{1}{2\alpha^2T{\rm Tr}(Q)}$. This concludes the proof of the exponential moment estimates for the exact solution of the stochastic Schr\"odinger equation \eqref{sse}.

Let us now prove the exponential moment estimates for the numerical solution \eqref{S1}. Let $p,q>1$ such that $1/p+1/q=1$, and set $\lambda=\frac{1}{2Tp{\rm Tr}(Q_\alpha)}$. Define $r_n=\lambda \exp(-\frac{n}{N})$ for $n=1,\ldots,N$, where $N\tau=T$, and introduce the filtration $\mathcal{F}_n=\sigma\{\delta W_k^{Q_\alpha}; k\le n-1\}$. Note that $u_n$ is $\mathcal{F}_n$-measurable. Let also $\tau^\star\in(0,p(p-1))$.

Using the definition of the scheme \eqref{S1} and H\"older's inequality, one has
\begin{align*}
\E\bigl[\exp\bigl(&r_{n+1}\norm{u_{n+1}}_{L^2}^2\bigr)~|~\mathcal{F}_n\bigr]\\
&\le \E[\exp\left(r_{n+1}\norm{u_n}_{L^2}^2\right)]
\left(\E[\exp(2pr_{n+1}{\rm Im}(\langle \Phi_{\tau}(u_n),\delta W_n^{Q_\alpha}\rangle))~|~\mathcal{F}_n]\right)^{\frac1p}\\
&\left(\E[\exp\left(qr_{n+1}\norm{\delta W_n^{Q_\alpha}}_{L^2}^2\right)]\right)^{\frac1q}.
\end{align*}
On the one hand, since $\delta W_n^{{Q_\alpha}}$ is a centered Gaussian random variable and by definition of $r_n$, one has
\begin{align*}
\E\left[\exp\left(qr_{n+1}\norm{\delta W^{{Q_\alpha}}_n}_{L^2}^2\right)\right]&
\le\left(1-2qr_{n+1}\E\left[\norm{\delta W_n^{{Q_\alpha}}}_{L^2}^2\right]\right)^{-\frac12}\\
&\le \left(1-2q\lambda\tau{\rm Tr}(Q_\alpha)\right)^{-\frac12},
\end{align*}
under the condition that $\tau<\frac{1}{2q\lambda {\rm Tr}(Q_\alpha)}=\frac{p^2}{q}$. This condition thus holds when $\tau<\tau^\star$.

On the other hand, conditional on $\mathcal{F}_n$, the random variable
$\langle \Phi_{\tau}(u_n),\delta W_n^{{Q_\alpha}}\rangle$ is also Gaussian and centered, thus
\begin{align*}
\E\left[\exp\left(2 pr_{n+1}{\rm Im}(\langle \Phi_{\tau}(u_n),\delta W^{{Q_\alpha}}_n\rangle)\right)~|~\mathcal{F}_n\right]
&\le \exp\left(2 p^2r_{n+1}^2{\rm Var}[\langle \Phi_{\tau}(u_n),\delta W^{{Q_\alpha}}_n\rangle]\right)\\
&\le \exp\left(2p^2\lambda r_{n+1}\tau{\rm Tr}(Q_{\alpha})\norm{u_n}_{L^2}^2\right).
\end{align*}
Gathering these estimates and taking expectation yield
\begin{align*}
\E\left[\exp\left(r_{n+1}\norm{u_{n+1}}_{L^2}^2\right)\right]&\le
\E\left[\exp\left(r_{n+1}(1+2p\lambda \tau{\rm Tr}(Q_\alpha))\norm{u_n}_{L^2}^2\right)\right]\\
&\left(1-2q\lambda\tau{\rm Tr}(Q_{\alpha})\right)^{-\frac{1}{2q}}.
\end{align*}
Having chosen $\lambda=\frac{1}{2pT{\rm Tr}(Q_\alpha)}$, one then gets $r_{n+1}(1+2p\lambda \tau{\rm Tr}(Q_\alpha))=r_n\e^{-\frac{\tau}{T}}(1+\frac{\tau}{T})\le r_n$.

A recursion on $n$ then gives the following estimate
\[
\underset{0\le n\tau\le T}\sup~\E\left[\exp\left(r_{n}\norm{u_{n}}_{L^2}^2\right)\right]\le \exp(\lambda\norm{u_{0}}_{L^2}^2)(1-2q\lambda\tau{\rm Tr}(Q_\alpha))^{-\frac{N}{2q}}\le C(\lambda,u_0)<\infty,
\]
for $\tau<\tau^\star$, where the quantity $C(\lambda,u_0)$ does not depend on $\tau$.

We are now in position to conclude the proof of exponential moments estimates for the numerical solution. Let $\mu$ such that $\mu <\frac{\e^{-1}}{2T{\rm Tr}(Q_\alpha)}$. Note that $r_N=\lambda \e^{-1}$, thus there exists $p>1$ such that $\mu\le r_N\le r_n$ for all $n\in\{0,\ldots,N\}$. This then implies that
\[
\underset{0\le n\tau\le T}\sup~\E\left[\exp\left(\mu \norm{u_n}_{L^2}^2\right)\right]\le C(\mu,T,Q,u_0)<\infty,
\]
for all $\tau\in(0,\tau^\star)$.

This concludes the proof of Theorem~\ref{thm-expm}.
\end{proof}

\subsection{Proofs of Propositions~\ref{prop:conv1/2} and~\ref{prop:conv1}}

Before we start with these proofs, it is convenient to introduce some auxiliary notation and provide the steps
that are common for both proofs. Define $w(t)=\displaystyle -\alpha\ii\int_{0}^{t}S(t-s)\,\dd W^Q(s)$ for all $t\ge 0$ and $w_n=\displaystyle-\alpha\ii\sum_{k=0}^{n-1}S(\tau)^{n-k}\delta W_k^Q$ for all $n\ge 0$. Introduce also $v(t)=u(t)-w(t)$ and $v_n=u_n-w_n$. Let $t_k=k\tau$. Recall that $S_n=\displaystyle\tau\sum_{k=0}^{n-1}\left(\norm{u(k\tau)}_{L^2}^2+\norm{u_k}_{L^2}^2\right)$.

Define $\epsilon_n=\norm{v(t_n)-v_n}_{L^2}$ and $e_n=\norm{u(t_n)-u_n}_{L^2}$.
Then the error between the numerical and exact solution reads $e_n\le \epsilon_n+\norm{w_n-w(t_n)}_{L^2}$.

Let us first deal with the error term $\norm{w_n-w(t_n)}_{L^2}$ for the stochastic convolution:
employing the It\^o isometry formula, with $\sigma=1$ (resp. $\sigma=2$) if Assumption~\ref{as1} (resp. Assumption~\ref{as2}) is satisfied, one has
\begin{align*}
\E\left[\norm{w_n-w(t_n)}_{L^2}^2\right]&=\alpha^2\E\left[\norm{\sum_{k=0}^{n-1}\int_{t_k}^{t_{k+1}}
\left(S(\tau)^{n-k}-S(t_{n}-t)\right)\,\dd W^Q(t)}_{L^2}^2\right]\\
&=\alpha^2\sum_{k=0}^{n-1}\int_{t_k}^{t_{k+1}}\norm{\left(S(\tau)^{n-k}-S(t_{n}-t)\right)Q^{\frac12}}_{\mathcal{L}_2^0}^2\,\dd t\\
&\le \alpha^2\sum_{k=0}^{n-1}\int_{t_k}^{t_{k+1}}\left|t-t_{k}\right|^\sigma\,\dd t\norm{Q^\frac12}_{\mathcal{L}_2^\sigma}^2\le C(T,\alpha,Q)\tau^\sigma,
\end{align*}
using properties of the semigroup $S$. Since the distribution of $w_n-w(t_n)$ is Gaussian, for every $q\in[1,\infty)$, there exists $C_q(T,Q)\in(0,\infty)$ such that one has
\begin{equation}\label{eq:stoconv}
\E[\norm{w_n-w(t_n)}_{L^2}^q]\le C_q(T,\alpha,Q)\tau^{\frac{q\sigma}{2}}.
\end{equation}

It remains to treat the error term $\epsilon_n=\norm{v_n-v(t_n)}_{L^2}$.

Using the mild formulation~\eqref{mild} of the solution $u(t_n)$ and the definition of
the splitting scheme~\eqref{S1} for $u_n$, one obtains
\begin{equation}\label{eq:errordecomp}
\begin{aligned}
v_{n+1}-v(t_{n+1})&=\left(S(\tau)v_n-\ii\int_{t_n}^{t_{n+1}}S(\tau)F(\Phi_{t-t_n}(u_n))\,\dd t\right)\\
&\quad-\left( S(\tau)v(t_n)-\ii\int_{t_n}^{t_{n+1}}S(t_{n+1}-t)F(u(t))\,\dd s\right)\\
&=S(\tau)\left(v_n-v(t_n)\right)-\ii\int_{t_n}^{t_{n+1}}\left( S(\tau)F(\Phi_{t-t_n}(u_n))-S(t_{n+1}-t)F(u(t)) \right)\,\dd t\\
&=S(\tau)\left(v_n-v(t_n)\right)+E_n^1+E_n^2+E_n^3+E_n^4,
\end{aligned}
\end{equation}
where
\begin{align*}
E_n^1&=\ii\int_{t_n}^{t_{n+1}}\left( S(t_{n+1}-t)-S(\tau)\right)F(u(t))\,\dd t\\
E_n^2&=\ii\int_{t_n}^{t_{n+1}}S(\tau)\left( F(u(t))-F(u(t_n)) \right)\,\dd t\\
E_n^3&=\ii\tau S(\tau)\left( F(u(t_n))-F(u_n) \right)\\
E_n^4&=\ii\int_{t_n}^{t_{n+1}}S(\tau)\left( F(u_n)-F(\Phi_{t-t_n}(u_n)) \right)\,\dd t.
\end{align*}

For the first term, using properties of the semigroup $S$ (see Lemma~\ref{lemma1}), for $\sigma\in\{1,2\}$, one has
\[
\norm{E_n^1}_{L^2}\le C\tau^{\frac{\sigma}{2}}\int_{t_n}^{t_{n+1}}\norm{F(u(t))}_{H^\sigma}\,\dd t.
\]
The treatment of the second term $E_n^2$ is different for the two propositions, details are provided below.

For the third term, recall that $\norm{u_n-u(t_n)}_{L^2}\le \epsilon_n+\norm{w_n-w(t_n)}_{L^2}$.
Using~\eqref{eq:KF}, one obtains
\begin{align*}
\norm{E_n^3}_{L^2}&\le \tau\left(C_F+K_F(\norm{u(t_n)}_{L^2}^2+\norm{u_n}_{L^2}^2)\right)\norm{u_n-u(t_n)}_{L^2}\\
&\le \tau\left(C_F+K_F(\norm{u(t_n)}_{L^2}^2+\norm{u_n}_{L^2}^2)\right)\epsilon_n\\
&\quad+\tau\left(C_F+K_F(\norm{u(t_n)}_{L^2}^2
+\norm{u_n}_{L^2}^2)\right)\norm{w_n-w(t_n)}_{L^2}.
\end{align*}
For the fourth term, using~\eqref{eq:KF}, the equality $\norm{\Phi_{t-t_n}(u_n)}_{L^2}=\norm{u_n}_{L^2}$, and Assumption~\ref{as3}, one obtains
\begin{align*}
\norm{E_n^4}_{L^2}&\leq C\int_{t_n}^{t_{n+1}}\left(1+\norm{u_n}^2_{L^2}+\norm{\Phi_{t-t_n}(u_n)}^2_{L^2}\right)\norm{u_n-\Phi_{t-t_n}(u_n)}_{L^2}\,\dd t\\
&\leq C\left(1+2\norm{u_n}_{L^2}^5\right)\int_{t_n}^{t_{n+1}}|t-t_n|\,\dd t\\
&\leq C\tau^2\left(1+\norm{u_n}_{L^2}^5\right).
\end{align*}

At this stage, it is necessary to treat separately the proofs for Proposition~\ref{prop:conv1/2} and~\ref{prop:conv1}.

\begin{proof}[Proof of Proposition~\ref{prop:conv1/2}]

Assume that $\sigma=1$. For the second error term $E_n^2$, using the assumption on $F$ and Cauchy--Schwarz inequality, one has
\[
\norm{E_n^2}_{L^2}^2\le C\int_{t_n}^{t_{n+1}}\left(1+\norm{u(t)}_{L^2}^2+\norm{u(t_n)}_{L^2}^2\right)^2\,\dd t
\int_{t_n}^{t_{n+1}}\norm{u(t)-u(t_n)}_{L^2}^2\,\dd t.
\]

Gathering all the estimates, and using the isometry property $\norm{S(\tau)\left(v_n-v(t_n)\right)}_{L^2}=\norm{v_n-v(t_n)}_{L^2}=\epsilon_n$, from~\eqref{eq:errordecomp} one obtains
\[
\epsilon_{n+1}\le \bigl(1+C_F\tau+K_F\tau\Gamma_n)\epsilon_n+R_n,
\]
where we define $\Gamma_n=\norm{u(t_n)}_{L^2}^2+\norm{u_n}_{L^2}^2$ and $R_n=\norm{E_n^1}_{L^2}+\norm{E_n^2}_{L^2}+
\norm{E_n^4}_{L^2}+K_F\tau\Gamma_n\norm{w_n-w(t_n)}_{L^2}$. Using a discrete Gronwall inequality and the equality $\epsilon_0=0$, one gets for all $n\in\{0,\ldots,N\}$
\[
\exp\left(-C_Fn\tau-K_F\tau\sum_{k=0}^{n-1}\Gamma_k\right)\epsilon_n\le \sum_{k=0}^{n-1}R_k.
\]
Rewriting $\displaystyle\tau\sum_{k=0}^{n-1}\Gamma_k=S_n$ and $\norm{u_n-u(t_n)}_{L^2}\leq\epsilon_n+\norm{w_n-w(t_n)}_{L^2}$,
applying Minkowskii's inequality yields for $q\in[1,\infty)$
\[
\E\left[\exp\left(-qK_FS_n\right)\norm{u_n-u(t_n)}_{L^2}^q\right]^{\frac{1}{q}}\le \e^{C_FT}
\sum_{k=0}^{n-1}\left(\E\left[R_k^q\right]\right)^{\frac{1}{q}}+\e^{C_FT}\left(\E\left[\norm{w_n-w(t_n)}_{L^2}^q\right]\right)^{\frac{1}{q}}.
\]
We now estimate each of the terms above. Let us first recall that $R_k=\norm{E_k^1}_{L^2}+\norm{E_k^2}_{L^2}
+\norm{E_k^4}_{L^2}+K_F\tau\Gamma_k\norm{w_k-w(t_k)}_{L^2}$. Using the triangle inequality, followed by Cauchy--Schwarz's inequality,
the assumption on the nonlinearity $F$ as well as moment estimates in the $L^2$ and $H^1$ norms for the exact solution (Corollary~\ref{cor:momentsL2} and Proposition~\ref{prop:exact}),
one obtains
\begin{align*}
\E\left[\norm{E_k^1}^q_{L^2}\right]^{1/q}
&\leq C\tau^{1/2}\int_{t_k}^{t_{k+1}}\E\left[\norm{F(u(t))}_{H^1}^q\right]^{1/q}\,\dd t\\
&\leq C\tau^{1/2}\int_{t_k}^{t_{k+1}}\E\left[\norm{u(t)}_{H^1}^{2q}\right]^{1/(2q)}
\E\left[(1+\norm{u(t)}_{L^2}^2)^{2q}\right]^{1/(2q)}\,\dd t\\
&\leq C\tau^{1/2}\int_{t_k}^{t_{k+1}}\,\dd t\leq C\tau^{3/2}.
\end{align*}
For the second term, we use Cauchy--Schwarz's inequality and moment bounds and regularity properties of the exact solution
from Proposition~\ref{prop:exact} to get
\begin{align*}
\E\left[\norm{E_k^2}^q_{L^2}\right]^{1/q}
&\leq C  \left( \int_{t_k}^{t_{k+1}}\E\left[ \left(1+\norm{u(t)}_{L^2}^2+\norm{u(t_k)}_{L^2}^2\right)^{2q}\right]^{1/q}\,\dd t \right)^{1/2}\\
&\left( \int_{t_k}^{t_{k+1}}\E\left[\norm{u(t)-u(t_k)}_{L^2}^{2q}\right]^{1/q}\,\dd t \right)^{1/2}\\
&\leq C \tau^{1/2}\left( \int_{t_k}^{t_{k+1}} |t-t_k|\,\dd t\right)^{1/2}\leq C\tau^{3/2}.
\end{align*}
Similarly, using the Cauchy--Schwarz's inequality and the moment estimates in the $L^2$ norm for the numerical solution (Corollary~\ref{cor:momentsL2}), we obtain
\begin{align*}
\E\left[\norm{E_k^4}^q_{L^2}\right]^{1/q}&\leq
C\tau^{2}\E\left[\left(1+2\norm{u_n}_{L^2}^2\right)^{2q}\right]^{1/(2q)}\E\left[\norm{u_n}^{10q}_{L^2}\right]^{1/(2q)}\\
&\leq C\tau^2.
\end{align*}
Thanks to the bounds for the moments in the $L^2$ norm given by Corollary~\ref{cor:momentsL2}, as well as to the error estimate~\eqref{eq:stoconv} for the stochastic convolution proved above, we obtain the estimate
\begin{align*}
\E\left[\left(K_F\tau\Gamma_k\norm{w_k-w(t_k)}\right)^q_{L^2}\right]^{1/q}
&\leq C\tau\E\left[\Gamma_k^{2q}\right]^{1/(2q)}\E\left[\norm{w_k-w(t_k)}_{L^2}^{2q}\right]^{1/(2q)}\\
&\leq C\tau\E\left[\norm{w_k-w(t_k)}_{L^2}^{2q}\right]^{1/(2q)}\leq C\tau\tau^{1/2}\leq C\tau^{3/2}.
\end{align*}
With all these estimates at hand, we arrive at
\[
\sum_{k=0}^{n-1}\left(\E\left[R_k^q\right]\right)^{\frac{1}{q}}\le C_q(T,u_0,\alpha,Q)\tau^{\frac12}.
\]
Finally, we obtain
\begin{align*}
\E\left[\exp\left(-qK_FS_n\right)\norm{u_n-u(t_n)}_{L^2}^q\right]^{\frac{1}{q}}&\le \e^{C_FT}
\sum_{k=0}^{n-1}\left(\E\left[R_k^q\right]\right)^{\frac{1}{q}}+\e^{C_FT}\left(\E\left[\norm{w_n-w(t_n)}_{L^2}^q\right]\right)^{\frac{1}{q}}\\
&\leq C_q(T,u_0,\alpha,Q)\tau^{\frac12}+C_q(T,\alpha,Q)\tau^{\frac12},
\end{align*}
using~\eqref{eq:stoconv} in the last step.

This concludes the proof of Proposition~\ref{prop:conv1/2}.
\end{proof}

We now turn to the proof of the second auxiliary result.

\begin{proof}[Proof of Proposition~\ref{prop:conv1}]
Assume that $\sigma=2$. As explained above, one requires to substantially modify the treatment of the error term $E_n^2$. As will be clear below, some changes in the analysis of the error $\epsilon_n$ are required too.

Using a second-order Taylor expansion of the nonlinearity $F$ and equation \eqref{eq:as2-1} (assumption on $F''$),
one obtains the decomposition $E_n^2=E_n^{2,1}+E_n^{2,2}$ where
\begin{align*}
E_n^{2,1}&=\ii \int_{t_n}^{t_{n+1}}S(\tau)F'(u(t_n)).\bigl(u(t)-u(t_n)\bigr)\,\dd t\\
\norm{E_n^{2,2}}_{L^2}&\le C\int_{t_n}^{t_{n+1}}\left(1+\norm{u(t_n)}_{L^2}+\norm{u(t)}_{L^2}\right)\norm{u(t)-u(t_n)}_{L^2}^2\,\dd t.
\end{align*}

Using the moment and increment bounds in the $L^2$ norm for the exact solution, see Proposition~\ref{prop:exact}, and the Cauchy--Schwarz inequality, one has $\bigl(\E[\norm{E_n^{2,2}}_{L^2}^q]\bigr)^{\frac{1}{q}}\le C\tau^2$.

In addition, using the mild formulation of the exact solution~\eqref{mild},
one has the decomposition $E_n^{2,1}=E_n^{2,1,1}+E_n^{2,1,2}+E_n^{2,1,3}$, where
\begin{align*}
E_n^{2,1,1}&=\ii \int_{t_n}^{t_{n+1}}S(\tau)F'(u(t_n)).\left(S(t-t_n)-I)\right)u(t_n)\,\dd t\\
E_n^{2,1,2}&=\ii \int_{t_n}^{t_{n+1}}S(\tau)F'(u(t_n)).\left(\int_{t_n}^{t}S(t-s)F(u(s))\,\dd s\right)\,\dd t\\
E_n^{2,1,3}&=\ii\alpha \int_{t_n}^{t_{n+1}}S(\tau)F'(u(t_n)).\left(\int_{t_n}^{t}S(t-s)\,\dd W^Q(s)\right)\,\dd t.
\end{align*}
Owing to Lemma~\ref{lemma1} and to equation~\eqref{eq:as2-1} in Assumption~\ref{as2},
the first and second terms above are treated as follows: one has
\[
\norm{E_n^{2,1,1}}_{L^2}\le C\tau^2\left(1+\norm{u(t_n)}_{L^2}^2\right)\norm{u(t_n)}_{H^2},
\]
and
\[
\norm{E_n^{2,1,2}}_{L^2}\le C\left(1+\norm{u(t_n)}_{L^2}^2\right)\tau \int_{t_n}^{t_{n+1}}\norm{F(u(s))}_{L^2} \,\dd s.
\]
Using the stochastic Fubini Theorem, the third term is written as
\begin{align*}
E_n^{2,1,3}&=\ii\alpha \int_{t_n}^{t_{n+1}}S(\tau)F'(u(t_n)).\bigl(\int_{t_n}^{t}S(t-s)\,\dd W^Q(s)\bigr)\,\dd t\\
&=\ii\alpha \int_{t_n}^{t_{n+1}}\left(S(\tau)F'(u(t_n)).\int_{s}^{t_{n+1}}S(t-s)\,\dd t\right)\,\dd W^Q(s)\\
&=\ii\alpha \int_{t_n}^{t_{n+1}}\Theta_n(s)\,\dd W^Q(s),
\end{align*}
where we have defined the quantity $\displaystyle\Theta_n(s)=S(\tau)F'(u(t_n)).\int_{s}^{t_{n+1}}S(t-s)\,\dd t$.

Applying It\^o's formula, one gets
\begin{equation}\label{eq:En213}
\E\left[\norm{E_n^{2,1,3}}_{L^2}^2\right]=\alpha^2\int_{t_{n}}^{t_{n+1}}\E\left[\norm{\Theta_n(s)Q^{\frac12}}_{\mathcal{L}_2^0}^2\right]\,\dd s\le C\tau^3
\end{equation}
using again~\eqref{eq:as2-1} from Assumption~\ref{as2} and the moment estimates in the $L^2$ norm
of the exact solution from Corollary~\ref{cor:momentsL2}.

However the estimate~\eqref{eq:En213} is not sufficient to directly obtain the required error estimate for $\epsilon_n$ as in the proof of Proposition~\ref{prop:conv1/2}.
Improving this estimate requires to modify the approach used above to deal with this error term.

Starting from~\eqref{eq:errordecomp}, one obtains for all $n\ge 0$
\[
v_n-v(t_n)=\sum_{k=0}^{n-1}S(\tau)^{n-k-1}\left(E_k^1+E_k^2+E_k^3+E_k^4\right).
\]
Recalling the decomposition $E_k^{2,1}=E_k^{2,1,1}+E_k^{2,1,2}+E_k^{2,1,3}$ and
using the above bounds for the term $E_k^3$
then yields
\begin{align*}
\epsilon_n&\le \tau\sum_{k=0}^{n-1}\left(C_F+K_F\Gamma_k\right)\epsilon_k\\
&+\tau\sum_{k=0}^{n-1}\left(C_F+K_F\Gamma_k\right)\norm{w(t_k)-w_k}_{L^2}\\
&+\sum_{k=0}^{n-1}\left(\norm{E_k^1}_{L^2}+\norm{E_k^4}_{L^2}\right)\\
&+\sum_{k=0}^{n-1}\norm{E_k^2-E_k^{2,1,3}}_{L^2}+\norm{\sum_{k=0}^{n-1}S(\tau)^{n-1-k}E_k^{2,1,3}}_{L^2}.
\end{align*}
Applying the Gronwall inequality to get an almost sure inequality, then using the Cauchy--Schwarz
and Minkowskii's inequalities, one obtains for all $n\ge 0$ and all $q\in[1,\infty)$
\begin{align*}
\e^{-C_Fn\tau}\left(\E[\e^{-qK_FS_n}\epsilon_n^q]\right)^{1/q}&\le \tau\sum_{k=0}^{n-1}
\left(\E\left[\left(C_F+K_F\Gamma_k\right)^{2q}\right]\right)^{1/(2q)}\left(\E\left[\norm{w(t_k)-w_k}_{L^2}^{2q}\right]\right)^{1/(2q)}\\
&\quad+\sum_{k=0}^{n-1}\left(\left(\E\left[\norm{E_k^1}_{L^2}^q\right]\right)^{1/q}+\left(\E\left[\norm{E_k^4}_{L^2}^q\right]\right)^{1/q}\right)\\
&\quad+\sum_{k=0}^{n-1}\left(\E\left[\norm{E_k^2-E_k^{2,1,3}}_{L^2}^{q}\right]\right)^{1/q}\\
&\quad+\left(\E\left[\norm{\sum_{k=0}^{n-1}S(\tau)^{n-1-k}E_k^{2,1,3}}_{L^2}^{2q}\right]\right)^{1/(2q)},
\end{align*}
where in the last term, we have used the inclusion $L^{2q}(\Omega)\subset L^{q}(\Omega)$.
We recall that the term $E_k^2-E_k^{2,1,3}$ can be written as
\[
E_k^2-E_k^{2,1,3}=E_k^{2,1,1}+E_k^{2,1,2}+E_k^{2,2}.
\]
Using the same arguments as in the proof of Proposition~\ref{prop:conv1/2} (in particular moment estimates of Proposition~\ref{prop:exact} and Corollary~\ref{cor:momentsL2}, and the error estimate~\eqref{eq:stoconv} for the stochastic convolution), the treatment of the error terms in the right-hand side is straightforward,
except for the last one which requires more details that we now present.

One has the identity
\[
\sum_{k=0}^{n-1}S(\tau)^{n-1-k}E_k^{2,1,3}=\sum_{k=0}^{n-1}\int_{t_k}^{t_{k+1}}S(\tau)^{n-1-k}\ii\alpha \Theta_k(s)\,\dd W^Q(s)
\]
and using the Burkholder--Davis--Gundy inequality one then obtains
\begin{align*}
&\left(\E\left[\norm{\sum_{k=0}^{n-1}S(\tau)^{n-1-k}E_k^{2,1,3}}_{L^2}^{2q}\right]\right)^{1/(2q)}\\
&\le C_q \left(\E\left[\Bigl(\sum_{k=0}^{n-1}\alpha^2\int_{t_k}^{t_{k+1}}\norm{S(\tau)^{n-1-k}\Theta_k(s)Q^{\frac12}}_{\mathcal{L}_2^0}^2 \,\dd s\Bigr)^q \right]\right)^{1/(2q)}\le C\tau,
\end{align*}
where the last upper bound follows from the definition of $\Theta_k(s)$, from~Assumption~\ref{as2} and from the moment bounds in the $L^2$ norm for the exact solution, see Proposition~\ref{prop:exact}.


Finally, recalling that $\norm{u_n-u(t_n)}_{L^2}\le \epsilon_n+\norm{w_n-w(t_n)}_{L^2}$,
gathering all these estimates and using the bounds on the error in the stochastic convolution~\eqref{eq:stoconv},
we obtain
\begin{align*}
\E\left[\exp\left(-qK_FS_n\right)\norm{u_n-u(t_n)}_{L^2}^q\right]^{\frac{1}{q}}&\leq C_q(T,u_0,Q)\tau.
\end{align*}

This concludes the proof of Proposition~\ref{prop:conv1}.

\end{proof}

\section{Numerical experiments}\label{sec-num}
We present some numerical experiments in order to support and illustrate
the above theoretical results. In addition, we shall compare
the behavior of the splitting scheme \eqref{S1}
(denoted by \textsc{Split} below) with the following time integrators
\begin{itemize}
\item the classical Euler--Maruyama scheme (denoted \textsc{EM})
$$
u_{n+1}=u_n-\ii\tau\Delta u_n-\ii\tau F(u_n)-\ii\alpha\delta W_n^Q.
$$
\item the classical semi-implicit Euler--Maruyama scheme (denoted \textsc{sEM})
$$
u_{n+1}=u_n-\ii\tau\Delta u_{n+1}-\ii\tau F(u_n)-\ii\alpha\delta W_n^Q.
$$
\item the stochastic exponential integrator from \cite{MR3771721} (denoted \textsc{sEXP})
$$
u_{n+1}=S(\tau)\left(u_n-\ii\tau F(u_n)-\ii\alpha\delta W_n^Q\right).
$$
\item the Crank--Nicolson--Euler--Maruyama (denoted \textsc{CN})
$$
u_{n+1}=u_n-\ii{\tau}\Delta u_{n+1/2}-\ii\tau F(u_{n})-\ii\alpha\delta W_n^Q,
$$
where $u_{n+1/2}=\frac12\left(u_n+u_{n+1}\right)$. This is a slight modification
of the Crank--Nicolson from \cite{MR2268663}.
\end{itemize}

\subsection{Trace formulas for the mass}
We consider the stochastic Schr\"odinger equation \eqref{sse} on the interval $[0,2\pi]$
with periodic boundary condition, the coefficient $\alpha=1$,
and a covariance operator with
$\left(\gamma_k\right)_{k\in\mathbb{Z}}=\left(\frac1{1+k^2}\right)_{k\in\mathbb{Z}}$
and $\left(e_k(x)\right)_{k\in\mathbb{Z}}=\left(\frac1{\sqrt{2\pi}}\e^{\ii  k x}\right)_{k\in\mathbb{Z}}$.
We consider the initial value $u_0=\frac{2}{2-\cos(x)}$
and the following nonlinearities:
$V(x)u=\frac{3}{5-4\cos(x)}u$ (external potential),
$\left(V\star|u|^2\right)u$ with $V(x)=\cos(x)$ (nonlocal interaction), $F(u)=+|u|^2u$ (cubic).
We refer to \cite[Theorem 3.4]{MR1954077} for a result on global existence of solutions
to the cubic case. We use a pseudo-spectral method with $N_x=2^8$ modes and
the above time integrators with time-step size $\tau=0.1$.

Figure~\ref{fig:mass} displays the evolution of the expected value of the mass on the time
intervals $[0,1]$ (external potential) and $[0,25]$ (other cases).
The expected values are approximated using $M=75000$ samples.
The exact trace formulas for the splitting scheme,
shown in Proposition~\ref{prop:mass}, can
be observed. The growth rates of the other schemes
are qualitatively different than this linear rate of the exact solution: observe for
instance the exponential drift of \textsc{EM} in the first plot,
the fact that \textsc{sEXP} seems to overestimate the linear drift and
the fact that \textsc{sEM} underestimates it.
The \textsc{CN} scheme performs relatively well,
except in the cubic case (not displayed),
where it should use a much smaller step-size in order not to explode.

\begin{figure}[h]
\begin{subfigure}{.49\textwidth}
  \centering
  \includegraphics[width=.67\linewidth]{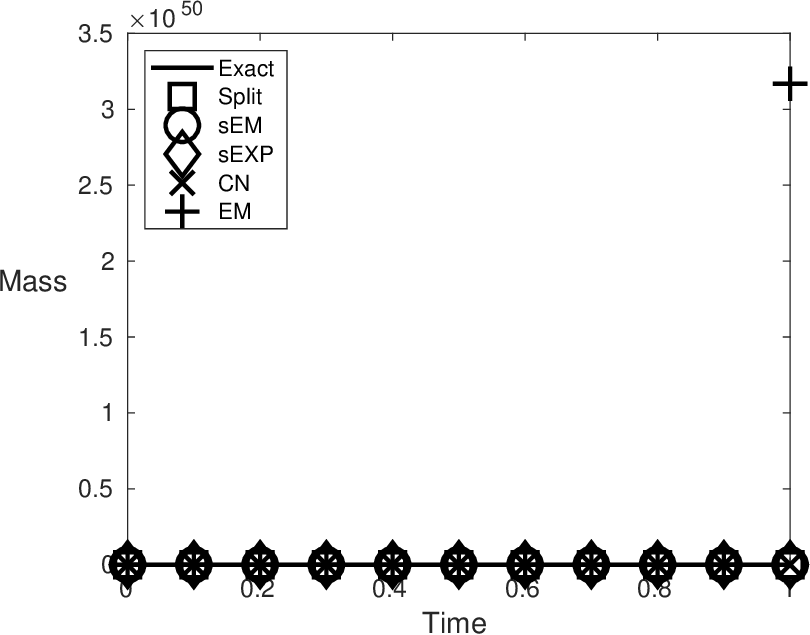}
  \caption{External potential}
\end{subfigure}%
\begin{subfigure}{.49\textwidth}
  \centering
  \includegraphics[width=.67\linewidth]{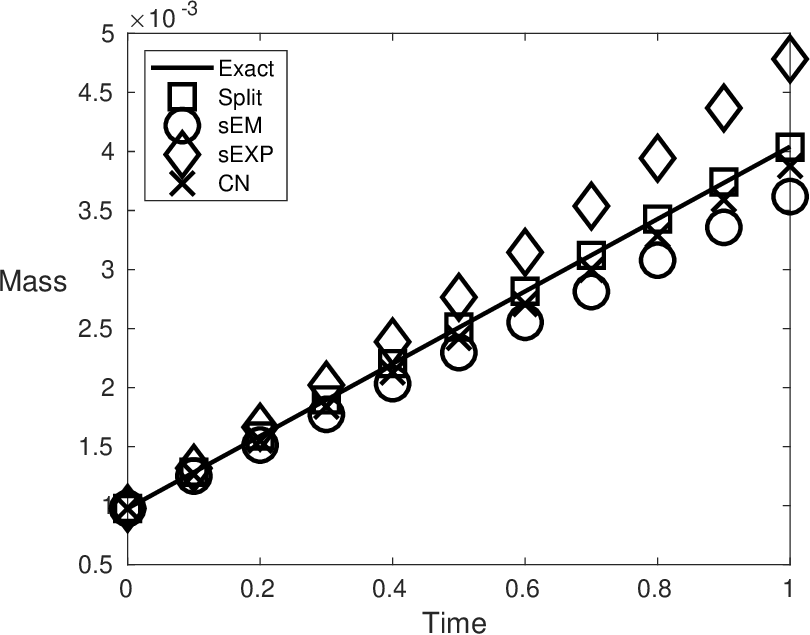}
  \caption{External potential}
\end{subfigure}
\begin{subfigure}{.49\textwidth}
  \centering
  \includegraphics[width=.67\linewidth]{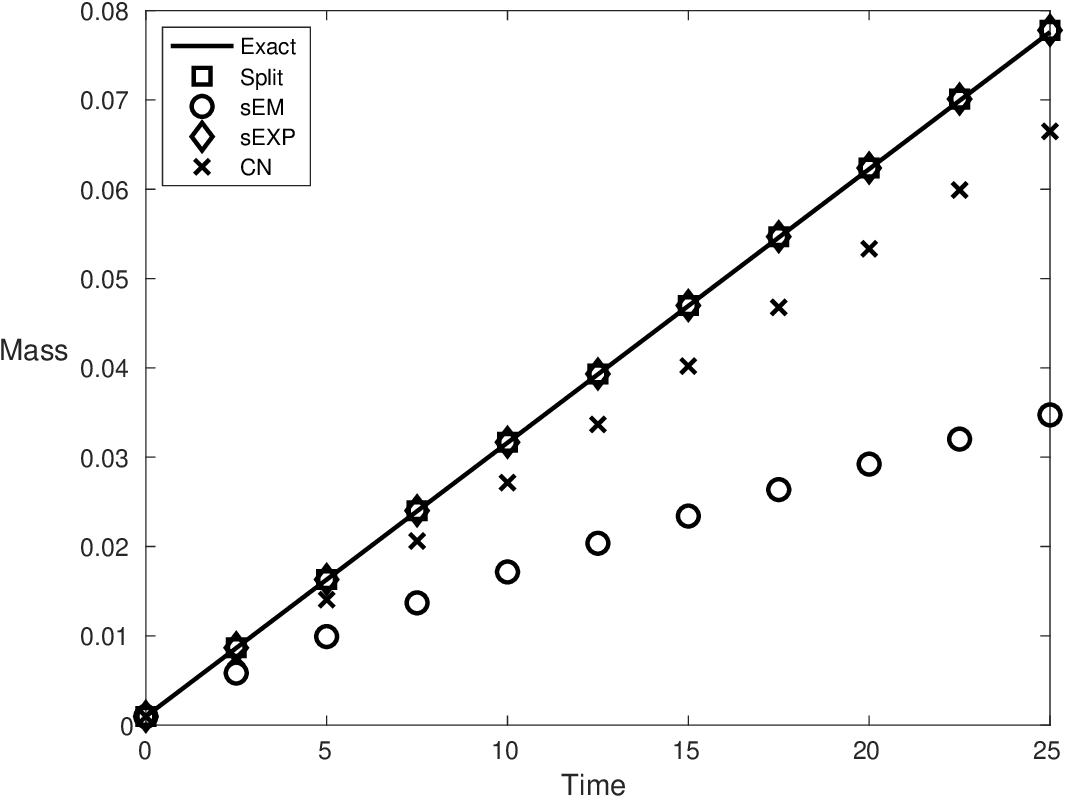}
  \caption{Nonlocal interaction}
\end{subfigure}
\begin{subfigure}{.49\textwidth}
  \centering
  \includegraphics[width=.67\linewidth]{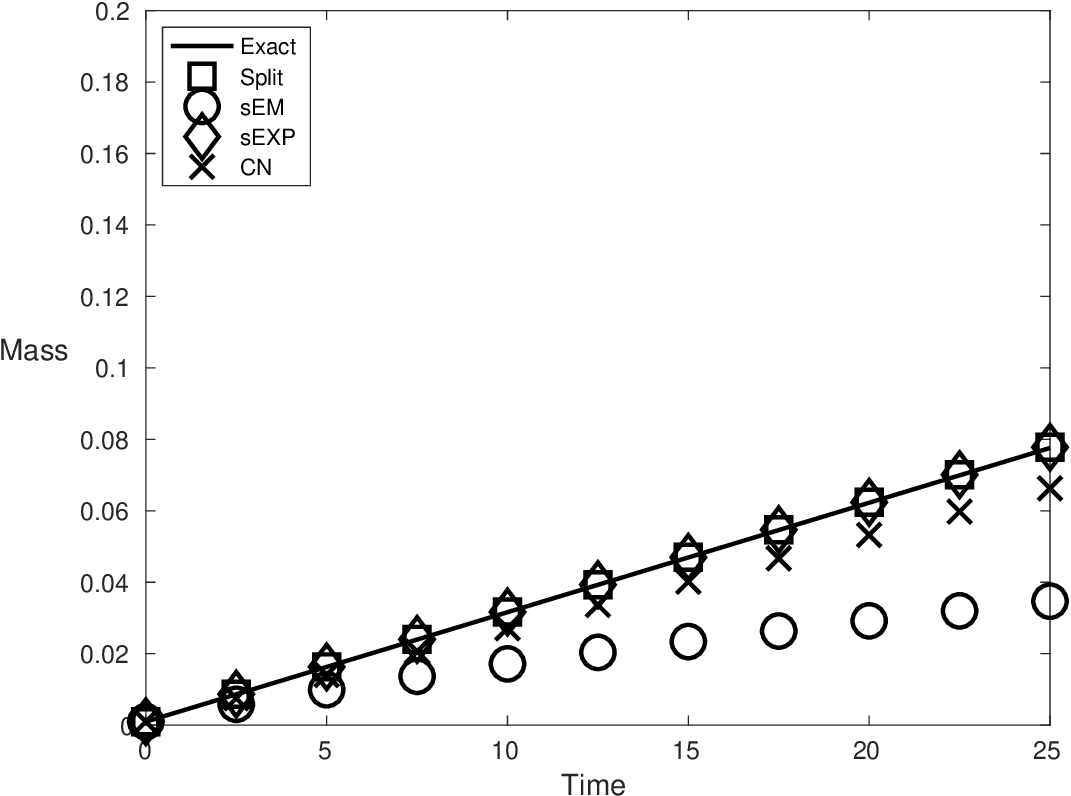}
  \caption{Cubic case}
\end{subfigure}
\caption{Trace formulas for mass of the splitting scheme (Split),
the Euler--Maruyama scheme (EM), the semi-implicit
Euler--Maruyama scheme (sEM), the exponential integrator (sEXP),
and the Crank--Nicolson (CN) schemes.}
\label{fig:mass}
\end{figure}

\subsection{Strong convergence}
In this subsection, we illustrate the strong convergence of the splitting scheme \eqref{S1}
as stated in Theorem~\ref{th:conv}.

To do this, we consider the stochastic Schr\"odinger
equation \eqref{sse} on the interval $[0,2\pi]$
with periodic boundary condition, and
a covariance operator with $\left(\gamma_k\right)_{k\in\mathbb{Z}}=\left(\frac1{1+k^2}\right)_{k\in\mathbb{Z}}$.
We consider the external potential $V(x)=\frac{3}{5-4\cos(x)}$ and nonlocal interaction
given by the potential $V(x)=\cos(x)$. We take the initial value $u_0=\frac{2}{2-\cos(x)}$ (external potential)
and $u_0=\frac{1}{1+\sin(x)^2}$ (nonlocal interaction). Additional parameters are:
coefficient in front of the noise $\alpha=1.5$,
time interval $[0,1]$, $250$ samples
used to approximate the expectations.
We use a pseudo-spectral method with $N_x=2^{10}$ modes and
the above time integrators. Strong errors, measured with $r=1$ at the end point,
are presented in Figure~\ref{fig:strong}. For this numerical experiment, the splitting and exponential integrators give
very close results. For clarity, only some of the values for the exponential integrator are displayed.
An order $1/2$ of convergence for the splitting scheme is observed. Note that, the strong order of convergence of
the other time integrators are not known in the case of the nonlocal interaction potential.
Observe that, in Figure~\ref{fig:strongb}, one sees that the order of convergence is less than $\frac12$,
however the exact value is not clearly visible. This may be due to numerical issues.
It may also happen that the order of convergence is not $\frac12$ due to the possible dependence of the order of convergence
with respect to the size of the noise and to the length of the time interval, see the conditions in Theorem~\ref{th:conv}.

\begin{figure}[h]
\begin{subfigure}{.5\textwidth}
  \centering
  \includegraphics[width=1.\linewidth]{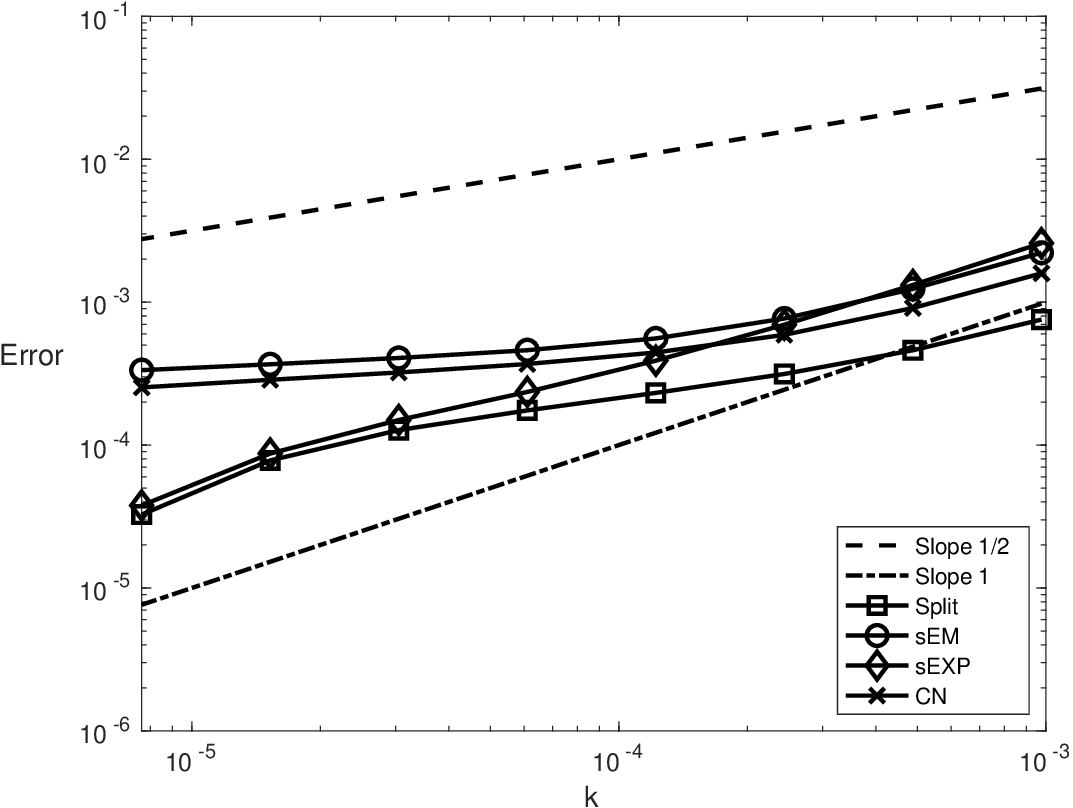}
  \caption{External potential}
\end{subfigure}%
\begin{subfigure}{.5\textwidth}
  \centering
  \includegraphics[width=1.\linewidth]{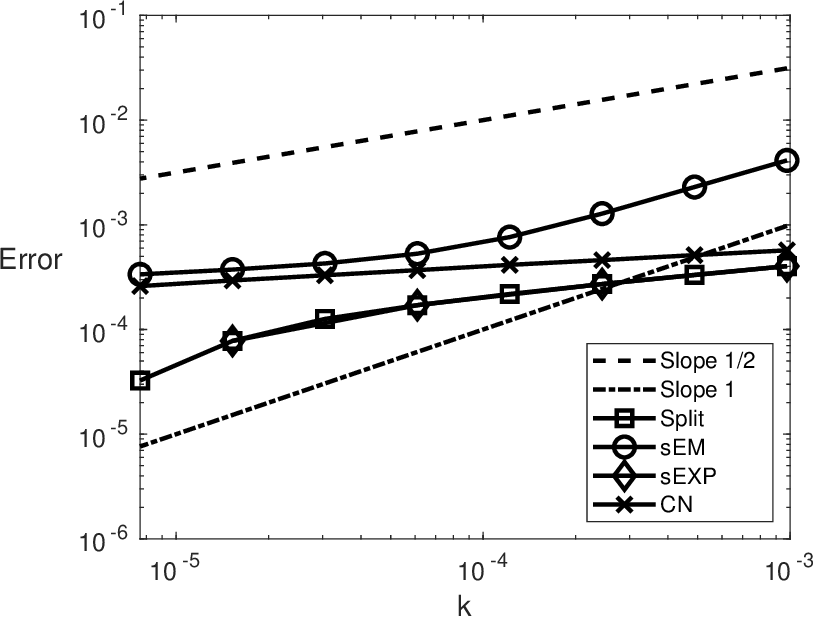}
  \caption{Nonlocal interaction}\label{fig:strongb}
\end{subfigure}
\caption{Strong errors for the stochastic Schr\"odinger equations.}
\label{fig:strong}
\end{figure}

In order to illustrate the higher order of convergence for the splitting scheme in the setting of Proposition~\ref{prop:conv1}, when $\sigma=2$,
we consider a smoother noise with covariance operator with
$\left(\gamma_k\right)_{k\in\mathbb{Z}}=\left(\frac1{1+k^4}\right)_{k\in\mathbb{Z}}$ (the other parameters for the simulation are as above).
In Figure~\ref{fig:strongS}, one observes that order of convergence $1$ may be obtained: this is indicated in Theorem~\ref{th:conv},
in the case $\sigma=2$ (smoother noise), if appropriate conditions are satisfied.

\begin{figure}[h]
\begin{subfigure}{.5\textwidth}
  \centering
  \includegraphics[width=1.\linewidth]{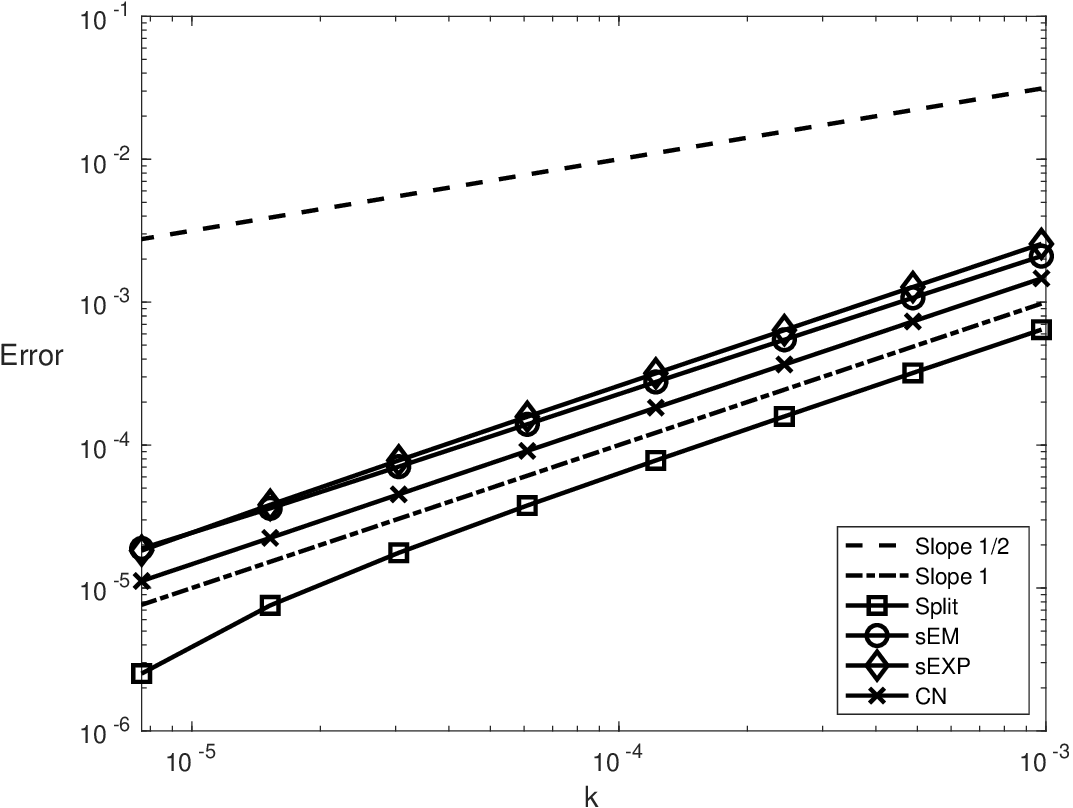}
  \caption{External potential}
\end{subfigure}%
\begin{subfigure}{.5\textwidth}
  \centering
  \includegraphics[width=1.\linewidth]{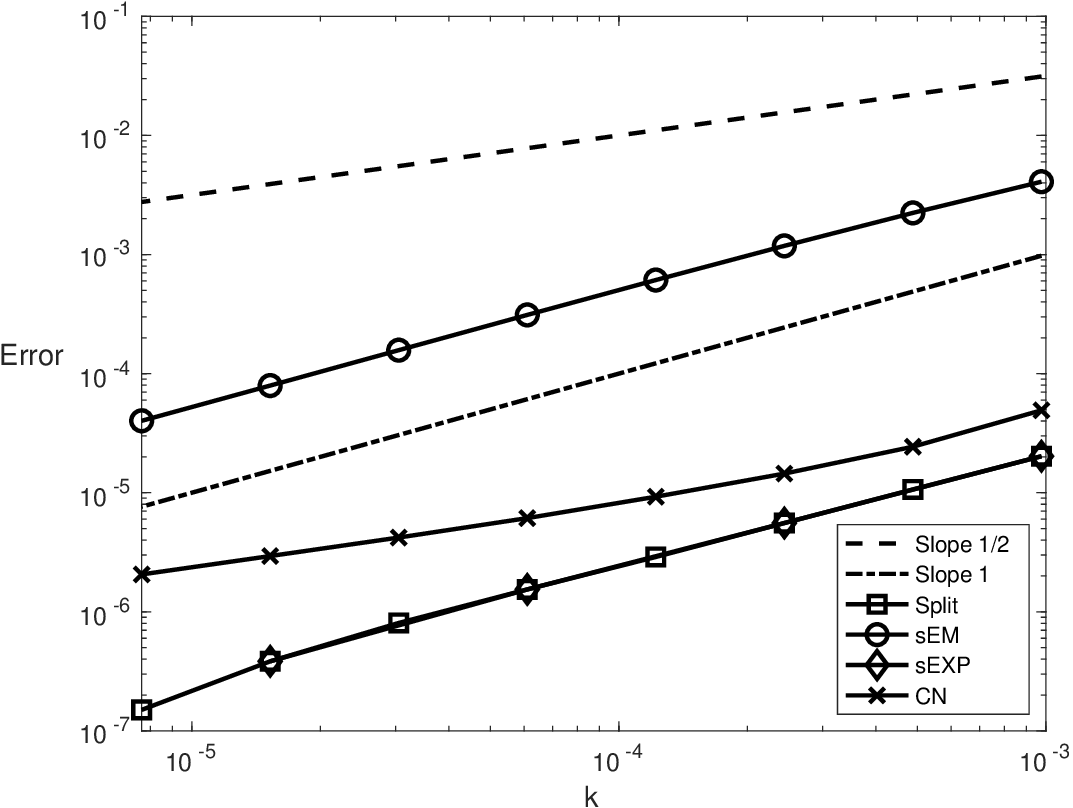}
  \caption{Nonlocal interaction}
\end{subfigure}
\caption{Strong errors for the stochastic Schr\"odinger equations with a smoother noise ($\sigma=2$).}
\label{fig:strongS}
\end{figure}

\subsection{Convergence in probability}
In this subsection we numerically demonstrate the order
of convergence in probability for the splitting scheme \eqref{S1}.
This order has been shown to be $1/2$ in Corollary~\ref{cor:pas} above.

Numerically, we investigate the order in probability by using the equation
\begin{equation}
\label{eq:probConv}
\max_{n\in\{1,2,\ldots,N\}} \norm{u_n  - u_{ref}(t_n)}_{L^2} \ge C \tau^\delta,
\end{equation}
where $u_{ref}$ denotes a reference solution computed using the splitting scheme with step-size $\tau_{ref}=2^{-16}$.
We then study the proportion of samples, $P$, fulfilling equation~\eqref{eq:probConv}
for given $C$ and $\delta$ and observe whether $P\to 0$ for the given $\delta$ as $\tau\to 0$ and $C$ increases.

We simulate $50$ samples of the splitting scheme applied to the SPDE \eqref{sse} with the initial value $u_0=\frac{2}{2-\cos(x)}$,
the nonlocal interaction and the same noise as in the previous subsection (non-smooth case).
In addition, we take the following parameters: $t\in[0,1]$, $N_x=2^8$ Fourier modes
and $\tau=2^n$ where $n=-6,-7,\ldots,-14$.
We then estimate the proportion $P$ of samples fulfilling \eqref{eq:probConv}
for each given $\tau$, $\delta= 0.4, 0.5, 0.6$, and $C=10^c$ for $c=1,2,3$.
The results are presented in Figure~\ref{fig:convProb}.

\begin{figure}[h!]
	\begin{center}
		\includegraphics*[width =.6\textwidth,keepaspectratio]{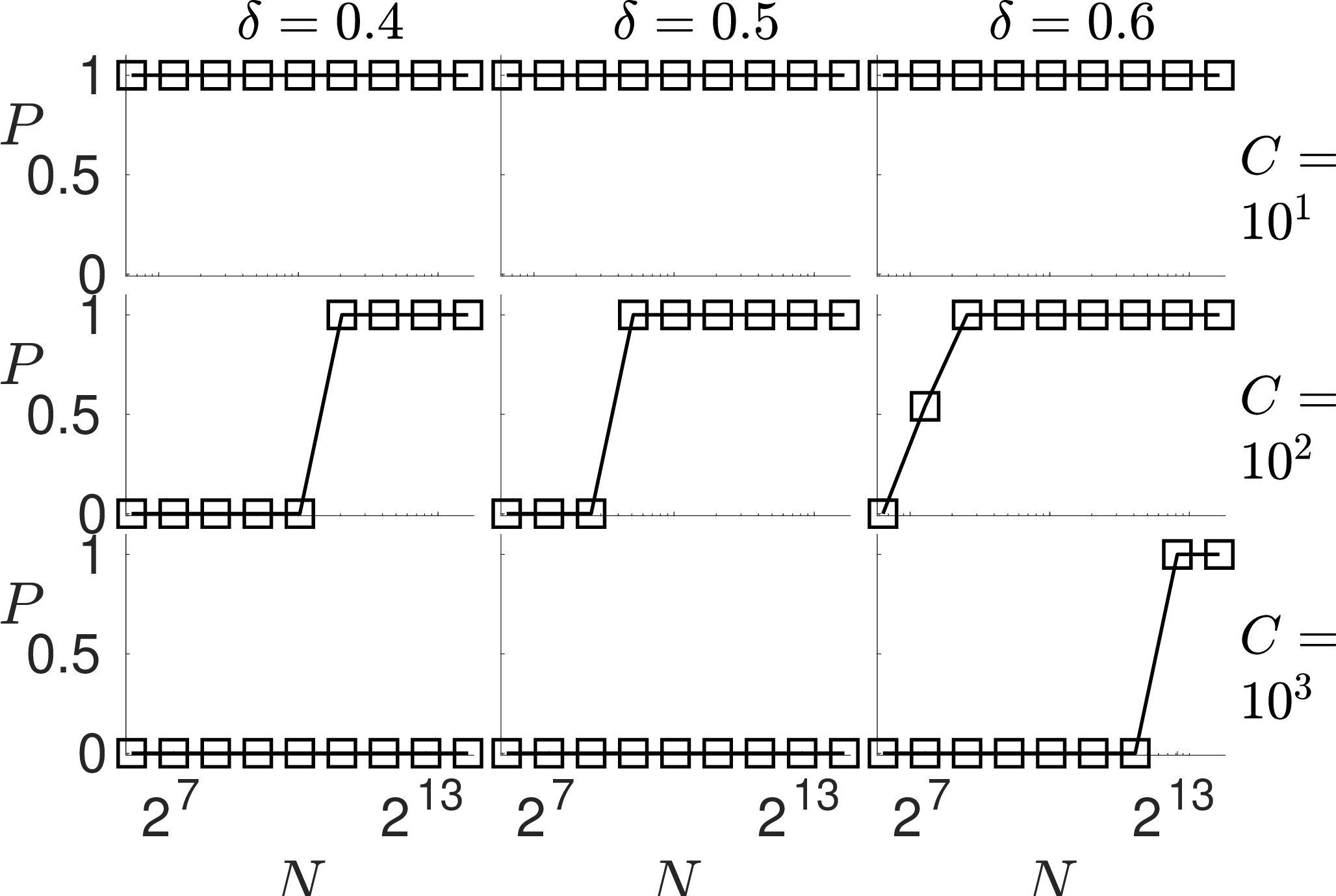}
		\caption{Proportion of samples fulfilling \eqref{eq:probConv} for
		         the splitting scheme ($N$ denotes the number of step-sizes).
			\label{fig:convProb}
		}
	\end{center}
\end{figure}

In this figure, one sees how the proportion of samples $P$
quickly goes to zero for $\delta \le 1/2$ and an increasing $C$.
Furthermore, this property does not hold for $\delta > 1/2$.
This numerical experiment thus confirms that the order
of convergence in probability of the splitting scheme is $1/2$,
as stated in Corollary~\ref{cor:pas}.

\subsection{Computational times}
In this numerical experiment, we compare the computational costs of the above time integrators
(expect the classical Euler--Maruyama scheme).
To do this, we consider the SPDE \eqref{sse} with the above nonlocal interaction potential
for times $t\in [0,2]$.
We discretize this SPDE using $N_x=2^{10}$ Fourier modes in space.
We run $100$ samples for each numerical scheme. For each scheme and each sample,
we run several time steps and compare the $L^2$ error
at the final time with a reference solution provided for the same sample by the same scheme
for a very small time-step $\tau=2^{-13}$.
Figure~\ref{fig:compcos} displays the total computational time for all the samples,
for each numerical scheme and each time-step, as a function of the averaged final error.
One observes better performance for the splitting scheme.

\begin{figure}[h]
\centering
\includegraphics*[height=5cm,keepaspectratio]{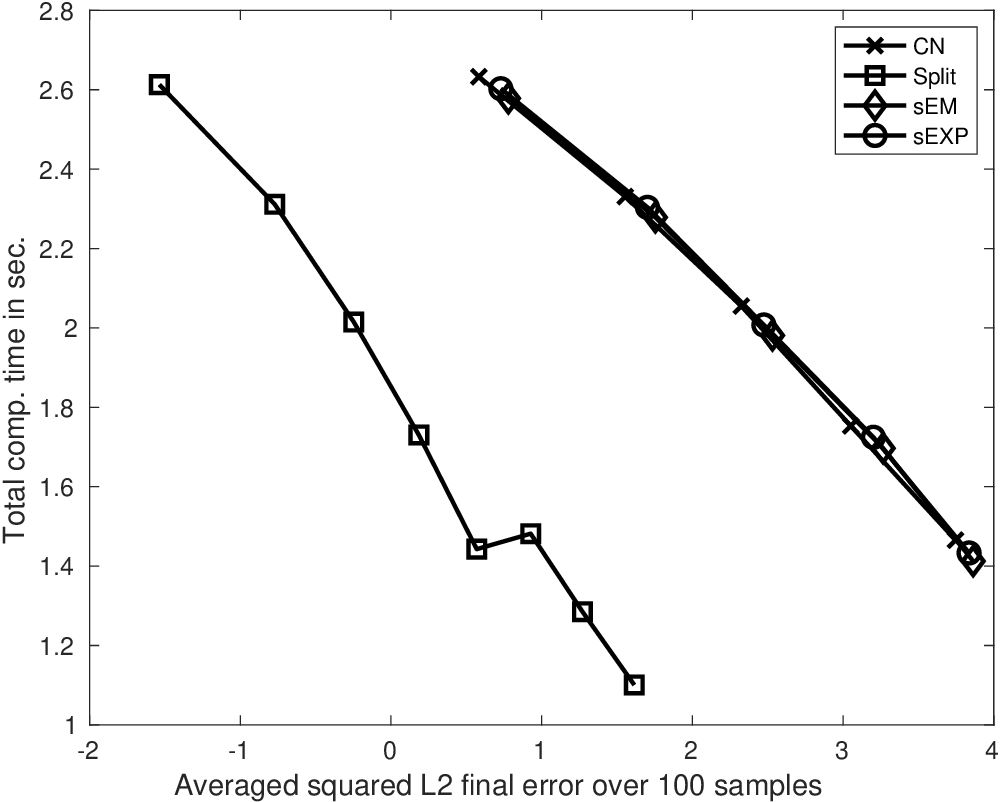}
\caption{Computational time as a function of the averaged final error for the four numerical methods.}
\label{fig:compcos}
\end{figure}

\appendix
\section*{Proof of Proposition~\ref{prop:exact}}\label{ap1}
This appendix provides the proofs of properties of the exact solution to \eqref{sse}.

\textbf{Global well-posedness}.
Let Assumption~\ref{as1} be satisfied. Since the nonlinearity $F$ is only locally Lipschitz continuous,
a truncation argument is used to prove global well-posedness.
Let us stress that the key property of~\eqref{sse} used in the argument below is the fact that $V[u]$ is real-valued.

Let $\theta\colon[0,\infty)\to [0,1]$ be a compactly supported Lipschitz continuous function,
such that $\theta(x)=1$ for $x\in[0,1]$. For any $R\in(0,\infty)$, set $V^R(u)=\theta(R^{-1}\norm{u}_{L^2})V[u]$
and $F^R(u)=V^R(u)u$. The mapping $F^R$ is globally Lipschitz continuous, and the SPDE
\[
\ii\mathrm{d}u^R(t)  = \Delta u^R(t)\, \mathrm{d}t + F^R(u^R(t))\,\mathrm{d}t+\,\alpha\mathrm{d}W^Q(t),
\]
with initial condition $u^R(0)=u_0$, thus admits a unique global solution denoted by $\bigl(u^R(t)\bigr)_{t\in[0,T]}$. Since the mapping $V^R$ is real-valued, the trace formula holds, see~\eqref{eq:trace}: indeed, one obtains
\[
\dd\norm{u^R(t)}_{L^2}^2=\alpha^2{\rm Tr}(Q)+2\alpha {\rm Im}(\langle u^R(t),\dd W^Q(t)\rangle).
\]
Taking expectation, one obtains the trace formula
$$
\E\left[\norm{u^R(t)}_{L^2}^2\right]=\norm{u_0}_{L^2}^2+2t\alpha^2{\rm Tr}(Q),
$$ where the right-hand side does not depend on truncation index $R$. Using the Burkholder--Davis--Gundy inequality, one obtains
\begin{align*}
\E[\underset{0\le t\le T}\sup~\norm{u^R(t)}_{L^2}^2]&\le 3\left(\norm{u_0}_{L^2}^2+\alpha^2T{\rm Tr}(Q)\right)\\
&\quad+3\alpha^2\int_0^T \sum_{k\in\N}|\gamma_k|^2\E[|\langle u^R(s),e_k\rangle|^2]\,\mathrm{d}s\\
&\le 3\left(\norm{u_0}_{L^2}^2+\alpha^2T{\rm Tr}(Q)\right)\\
&\quad+3\alpha^2\norm{Q^{\frac12}}_{\mathcal{L}_2^0}^2\int_0^T \E[\norm{u^R(s)}_{L^2}^2]\,\mathrm{d}s\\
&\le C(T,Q,u_0),
\end{align*}
where one observes that $C(T,Q,u_0)$ does not depend on $R$,
using the trace formula above for the term in the integral $\norm{u^R(s)}_{L^2}^2$.

Setting the truncation argument is then straightforward. Let $\tau^R=\inf\{t\ge 0; \norm{u^R(t)}_{L^2}>R\}$. If $R_1,R_2\ge R$, then $u^{R_1}(t)=u^{R_2}(t)$ for all $t\le \tau^R$, by construction of $F^R$. This allows us to define $u(t)$ solving~\eqref{sse} for all $t\in[0,\tau)$, where $\tau=\underset{R\to \infty}\lim~\tau^R$. Finally, $\tau=\infty$ almost surely, indeed for every $T\in(0,\infty)$, one has
\begin{align*}
\mathbb{P}(\tau\le T)&=\underset{R\to \infty}\lim~\mathbb{P}(\tau^R\le T)=\underset{R\to \infty}\lim~\mathbb{P}(\underset{0\le t\le T}\sup~\norm{u^R(t)}_{L^2}^2\ge R^2)\\
& \le \underset{R\to\infty}\lim~\frac{C(T,Q,u_0)}{R^2}=0,
\end{align*}
using the moment estimate above. This concludes the proof of the global well-posedness of~\eqref{sse}.

\textbf{Moment estimates in $H^1$}.
Next, let us prove the moment bounds for the exact solution to \eqref{sse}.
We provide details only for the moment estimates in the $H^1$ norm (under Assumption~\ref{as1})
$$
\underset{0\le t\le T}\sup~\E[\norm{\nabla u(t)}_{L^2}^{2p}]\le C_p(T,Q,u_0).
$$
Indeed, the moment estimates for the $L^2$ norm, namely
$$
\underset{0\le t\le T}\sup~\E[\norm{u(t)}_{L^2}^{2p}]\le C_p(T,Q,u_0),
$$
can either be obtained using similar arguments, or be deduced from the exponential moment estimates for which a detailed proof is provided above. Likewise, the proof of moment estimates
$$
\underset{0\le t\le T}\sup~\E[\norm{\nabla u(t)}_{H^2}^{2p}]\le C_p(T,Q,u_0)
$$
under Assumption~\ref{as2} would follow from similar arguments.

Let us first consider $\psi(u)=\norm{\nabla u}_{L^2}^{2}$ for all $u\in H^1$. Its first and second order derivatives are given by
\begin{align*}
\psi'(u).h&=2{\rm Re}(\langle \nabla u,\nabla h\rangle)\\
\psi''(u).(h,k)&=2{\rm Re}(\langle \nabla h,\nabla k\rangle)
\end{align*}
for $h,k\in H^1$. Using It\^o's formula, one gets
\begin{align*}
\mathrm{d}\norm{\nabla u(t)}_{L^2}^{2}&=\mathrm{d}\psi(u(t))\\
&=\psi'(u(t)).\mathrm{d}u(t)+\frac{\alpha^2}{2}\sum_{k\in\N}\psi''(u(t)).(\gamma_ke_k,\gamma_ke_k)\,\mathrm{d}t\\
&=2{\rm Im}(\langle\nabla\overline{u}(t),\nabla\Delta u(t)\rangle)\,\mathrm{d}t
+2{\rm Im}(\langle \nabla \overline{u}(t),\nabla F(u(t))\rangle)\,\mathrm{d}t\\
&~+2\alpha{\rm Im}(\langle \nabla u(t),\nabla \mathrm{d}W^Q(t)\rangle)~+\alpha^2\sum_{k\in\N}|\gamma_k|^2\norm{\nabla e_k}_{L^2}^2.
\end{align*}
The first term in the last equality vanishes, and when taking expectation the third term also vanishes. Using the condition~\eqref{eq:as1-2} to deal with the second term, one obtains
\[
\frac{\mathrm{d} \E[\norm{\nabla u(t)}_{L^2}^2]}{\mathrm{d}t}\le
C\left(1+\E[\norm{\nabla u(t)}_{L^2}^2]+\E[P_1(\norm{u(t)}_{L^2}^2)]\right),
\]
where $P_1$ is a polynomial mapping, see equation~\eqref{eq:as1-2}. Note that one has the upper bound $\underset{0\le t\le T}\sup~\E[P_1(\norm{u(t)}_{L^2}^2)]\le C(T,Q,u_0)$ due to moment bounds in the $L^2$ norm. Using the Gronwall Lemma then yields
\[
\underset{0\le t\le T}\sup~\E[\norm{\nabla u(t)}_{L^2}^2]\le C(T,Q,u_0).
\]
Let $p\ge 1$, then applying It\^o's formula for $\psi_p(u)=\psi(u)^p$ yields
\begin{align*}
\frac{\mathrm{d}\E[\norm{\nabla u(t)}_{L^2}^{2p}]}{\mathrm{d}t}&\le
C_p\norm{\nabla u(t)}_{L^2}^{2(p-1)}{\rm Im}(\langle \nabla u(t),\nabla F(u(t))\rangle)\\
&+C_p\alpha^2\sum_{k\in\N}|\gamma_k|^2\norm{\nabla e_k}_{L^2}^2\E[\norm{\nabla u(t)}_{L^2}^{2(p-1)}]\\
&\le C\left(1+\E[\norm{\nabla u(t)}_{L^2}^{2p}]+C\E[P_1(\norm{u(t)}_{L^2}^{2})^{p}]\right)
\end{align*}
using~\eqref{eq:as1-2} and Young's inequality. Using Gronwall's lemma then concludes the proof of the moment bounds in the $H^1$ norm.

\textbf{Temporal regularity}.
It remains to deal with the temporal regularity estimate. Using the mild formulation~\eqref{mild}, for any $0\leq t_1<t_2\leq T$, one has
\begin{align*}
u(t_2)-u(t_1)&=S(t_1)\bigl(S(t_2-t_1)-I\bigr)u_0\\
&~-\ii\int_0^{t_1}S(t_1-s)\bigl(S(t_2-t_1)-I\bigr)F(u(s))\,\mathrm{d}s\\
&~-\ii\int_{t_1}^{t_2}S(t_2-s)F(u(s))\,\mathrm{d}s\\
&~-\ii\alpha\int_{0}^{t_1}S(t_1-s)\bigl(S(t_2-t_1)-I\bigr)\,\mathrm{d}W^Q(s)\\
&~-\ii\alpha\int_{t_1}^{t_2}S(t_2-s)\,\mathrm{d}W^Q(s).
\end{align*}
Using Lemma~\ref{lemma1}, the first estimate of~\eqref{eq:as1-2} and the moment bounds in the $L^2$ and $H^1$ norms, one obtains
\begin{gather*}
\norm{S(t_1)\bigl(S(t_2-t_1)-I\bigr)u_0}_{L^2}\le C|t_2-t_1|^\frac12 \norm{u_0}_{H^1}\\
\E\left[\norm{\int_0^{t_1}S(t_1-s)\bigl(S(t_2-t_1)-I\bigr)F(u(s))\,\mathrm{d}s}_{L^2}^{2p}\right]\\
\le T^{2p-1}|t_2-t_1|^{p}\int_{0}^{T}\E\left[\norm{F(u(s))}_{H^1}^{2p}\right]\,\mathrm{d}s
\le C|t_2-t_1|^{\frac{2p}{2}}\\
\E\left[\norm{\int_{t_1}^{t_2}S(t_2-s)F(u(s))\,\mathrm{d}s}_{L^2}^{2p}\right]\le
|t_2-t_1|^{2p}\int_0^T\E\left[\norm{F(u(s))}_{L^2}^{2p}\right]\,\mathrm{d}s\\
\le C|t_2-t_1|^{2p}.
\end{gather*}
Using It\^o's isometry formula and Lemma~\ref{lemma1}, one has
\begin{gather*}
\E\left[\norm{\int_{0}^{t_1}S(t_1-s)\bigl(S(t_2-t_1)-I\bigr)\,\mathrm{d}W^Q(s)}_{L^2}^2\right]=t_1\norm{\bigl(S(t_2-t_1)-I\bigr)Q^{\frac12}}_{\mathcal{L}_2^0}^2\\
\le Ct_1|t_2-t_1| \norm{Q^{\frac12}}_{\mathcal{L}_2^1}^2\\
\E\left[\norm{\int_{t_1}^{t_2}S(t_2-s)\,\mathrm{d}W^Q(s)}^2\right]=|t_2-t_1|\norm{Q^\frac12}_{\mathcal{L}_2^{0}}.
\end{gather*}
Since the stochastic integrals have Gaussian distribution, gathering the estimates above yields
$$
\E\left[\norm{u(t_2)-u(t_1)}_{L^2}^{2p}\right]\le C_p(T,Q,u_0)|t_2-t_1|^p,
$$
for all $p\ge 1$ and $t_1,t_2\in[0,T]$. This concludes the proof of Proposition~\ref{prop:exact}.

\section*{Acknowledgements}
We thank Andr\'e Berg (Ume{\aa} University) for discussions on
the implementation of Figure~\ref{fig:convProb}.
The work of CEB was partially supported by the SIMALIN project ANR-19-CE40-0016 of the French National Research Agency.
The work of DC was partially supported by the Swedish Research Council (VR) (projects nr. 2018-04443).
Part of the work of CEB was carried out when working at Institut Camille Jordan, CNRS and Universit\'e Lyon 1.
Part of the work of DC was carried out when working for the Department of Mathematics and Mathematical Statistics at Ume{\aa} University.
The computations were performed on resources provided by the Swedish National Infrastructure
for Computing (SNIC) at HPC2N, Ume{\aa} University and at
Chalmers Centre for Computational Science and Engineering.

\bibliographystyle{plain}
\bibliography{labib}


%
%
%
\end{document}